\documentclass[12pt]{article}

\usepackage[latin1]{inputenc}
\usepackage{fullpage}
\usepackage{amsfonts}
\usepackage{amssymb}
\usepackage{amsmath}
\usepackage[mathscr]{euscript}
\usepackage{wasysym}
\usepackage{tikz-cd}
\usepackage{stmaryrd}
\usepackage{color}
\usepackage[colors]{optsys}

\renewcommand{\depart}{u} 
\renewcommand{\Depart}{U} 
\renewcommand{\DEPART}{\mathbb{U}} 
\renewcommand{\arrivee}{v} 
\renewcommand{\Arrivee}{V} 
\renewcommand{\ARRIVEE}{\mathbb{V}} 

\newcommand{\DEPARTbis}{\mathbb{W}} 
\newcommand{\arriveebis}{w}

\newcommand{\ARRIVEEbis}{\mathbb{W}}

\renewcommand{\arriveeter}{w}

\renewcommand{\ARRIVEEter}{\mathbb{W}}
 \newcommand{\FonctionDepartArrivee}{h}
\newcommand{\LinearQuadraticFonctionuncertain}{q}
\newcommand{\SquareMapping}{s}
\newcommand{\DirectionRelation}{\mathfrak{D}}

\title{Conditional Infimum and Hidden Convexity\\
in Optimization}

\author{Jean-Philippe Chancelier and Michel De Lara, \\ 
  CERMICS, Ecole des Ponts, Marne-la-Vall\'ee, France}

\begin{document}

\maketitle

\begin{abstract}
  Detecting hidden convexity is one of the tools to address nonconvex
  minimization problems.
  After giving a formal definition of hidden convexity, 
  we introduce the notion of conditional infimum, as it will prove instrumental in
  detecting hidden convexity. 
  We develop the theory of the conditional infimum, 
  and we establish a tower property, relevant for minimization problems. 
  Thus equipped, we provide a sufficient condition for 
  hidden convexity in nonconvex minimization problems. We illustrate our result on
  nonconvex quadratic minimization problems.
  We conclude with perspectives for using
  the conditional infimum in relation to
  the so-called S-procedure, to couplings and conjugacies,
  and to lower bound convex programs.
\end{abstract}



\section{Introduction}

Convex minimization problems display well-known features that make their
numerical resolution appealing. In particular, 
convex minimization algorithms are known to be simpler and less computationally
intensive, in comparison with nonconvex ones.
Thus, it is tempting to ``convexify'' a problem in order to solve it,
rather than to use nonconvex optimization.
More generally, it has long been searched how to relate a nonconvex minimization problem
to a convex one.
If the original nonconvex minimization problem is formulated
on a convex set, then the convex lower envelope of the objective function 
has the same minimum and a solution (argmin) of the original nonconvex problem 
is solution of the convex lower envelope problem \cite[Proposition~11]{Horst:1984}.
Needless to say that computing the lower envelope can be at least as difficult as
solving the original nonconvex problem.
This is why other approaches have been developed, like convexification 
by domain or range transformation as exposed
in \cite{Horst:1984} which provides a survey.
The vocable of ``hidden convexity'' covers different approaches:
duality and biduality analysis like
in~\cite{BenTal-Ben-Teboulle:1996};
identifying classes of nonconvex optimization problems whose convex
relaxations have optimal solutions which at the same time are global optimal
solutions of the original nonconvex
problems~\cite{Ben-Tal-den-Hertog-Laurent:2011}.
A survey of hidden convex optimization can be found in \cite{XiaYong:2020},
with its focus on three widely used ways to reveal the hidden
convex structure for different classes of nonconvex optimization problems. 
In this paper, we propose a definition of ``hidden convexity''
and a new way to reveal it by means of what we call the ``conditional infimum''.
We discuss these two points now.

Regarding hidden convexity, we consider a set~$\UNCERTAIN$, 
a function \( \fonctionuncertain : \UNCERTAIN \to \barRR \)
and a subset \( \Uncertain \subset \UNCERTAIN \).
We say that the minimization problem\footnote{%
In this paper, we address hidden convexity in \emph{optimization}.
In~\cite{Chancelier-DeLara:2021_ECAPRA_JCA}, we dealt with the stronger
notion of hidden convexity in a \emph{function}, that we characterized by means
of one-sided linear conjugacies.
Let \( \theta: \UNCERTAIN \to \PRIMAL \) be a mapping.
We say that the function \( \fonctionuncertain : \UNCERTAIN \to \barRR \)
displays hidden convexity with respect to the mapping~\( \theta \)
if there exists a vector space~$\PRIMAL$ and 
a convex function \( \fonctionprimal: \PRIMAL \to \barRR \) 
such that \( \fonctionuncertain = \fonctionprimal \circ \theta \).
Thus, hidden convexity in a function resorts to a property of \emph{convex factorization}.
\label{ft:hidden_convexity_function}}
\( \min_{\uncertain \in \Uncertain }
\fonctionuncertain\np{\uncertain} \) displays \emph{hidden convexity}
if there exists a vector space~$\PRIMAL$, 
a convex function \( \fonctionprimal: \PRIMAL \to \barRR \) 
and a convex subset \( \Convex \subset \PRIMAL \) such that 
\( \min_{\uncertain \in \Uncertain }
\fonctionuncertain\np{\uncertain} =
\min_{\primal \in \Convex} \fonctionprimal\np{\primal} \).
So, in the definition we propose, the original minimization problem
is not formulated on a vector space, even less on a convex domain

Now for the conditional infimum.
The operation of marginalization (that is,
partial minimization as in \cite[Theorem~5.3]{Rockafellar:1970})
is widely used in optimization, especially in the context of
studying how optimal values and optimal solutions depend on the
parameters in a given problem.
Another operation through which new functions are constructed by minimization is
the so-called epi-composition, developed by Rockafellar
(see \cite[p.~27]{Rockafellar-Wets:1998} and the historical note in
\cite[p.~36]{Rockafellar-Wets:1998});
epi-composition is called 
infimal postcomposition in \cite[p.~214]{Bauschke-Combettes:2017}.
Notice that the vocable of marginalization hinges at a corresponding operation
(of partial integration) in probability theory.
A nice parallelism between optimization and probability theories
has been pointed out by several authors
\cite{DelMoral:1997,Akian-Quadrat-Viot:1998}. 
Following this approach, we have relabelled epi-composition
as \conditionalinfimum\
in \cite[Definition~2.4]{Chancelier-DeLara:2021_ECAPRA_JCA},
with the notation \( \ConditionalInfimum{\theta}{\fonctionprimal} \).
The  expression ``conditional infimum'' appears in the conclusion part of
 \cite{Witsenhausen:1975b}, where it is defined with respect to a partition field, that is, a subset of the power set which
 is closed \wrt\ (with respect to) union and intersection, countable or not;
 however, the theory is not developed.
 Related notions can be found --- but defined on a measurable space
 equipped with a unitary Maslov measure ---
 in the following works:
 in \cite{DelMoral:1997}, the theory of performance is sketched
 and the ``conditional performance'' is defined;
 in \cite{Akian-Quadrat-Viot:1998}, the ``conditional cost excess'' is defined;
 in \cite{Barron-Suprema-Jensen:2003}, the ``conditional essential supremum''  is defined.  
 In this paper, we define the conditional infimum with respect to a
 correspondence between two sets, without requiring measurable structures,
 and we study its properties in the perspective
 of applications to optimization.
 \medskip

The paper is organized as follows.
In Sect.~\ref{Conditional_infimum_with_respect_to_a_correspondence},
we provide a definition of the
conditional infimum (and supremum) of a function with respect to a
correspondence (between two sets), followed by examples and main properties.
In Sect.~\ref{Applications_of_the_conditional_infimum_to_minimization_problems}
we develop applications of the conditional infimum to minimization problems.
In
Sect.~\ref{Detecting_hidden_convexity_in_optimization_problems_using_conditional_infimum},
we provide a sufficient condition for 
  hidden convexity in nonconvex minimization problems, and we illustrate our result on
  nonconvex quadratic minimization problems.
In the concluding Sect.~\ref{Conclusion}, we point out perspectives for using
  the conditional infimum in other contexts, namely in relation to
  the so-called S-procedure, to couplings and conjugacies,
  and to lower bound convex programs.

\section{Conditional infimum with respect to a correspondence}
\label{Conditional_infimum_with_respect_to_a_correspondence}

In~\S\ref{Definitions_of_the_conditional_infimum_and_supremum}, we provide a definition of the
conditional infimum (and supremum) of a function with respect to a
correspondence between two sets, followed by examples
in~\S\ref{Examples_of_the_conditional_infimum_and_supremum}.
Then, we expose properties of the conditional infimum and supremum
in~\S\ref{Properties_of_the_conditional_infimum_and_supremum}.

As we manipulate functions with values in~$\barRR = [-\infty,+\infty] $,
we adopt the Moreau \emph{lower ($\LowPlus$)} and \emph{upper ($\UppPlus$) additions} \cite{Moreau:1970},
which extend the usual addition~($+$) with 
\( \np{+\infty} \LowPlus \np{-\infty}=\np{-\infty} \LowPlus \np{+\infty}=-\infty \) and
\( \np{+\infty} \UppPlus \np{-\infty}=\np{-\infty} \UppPlus \np{+\infty}=+\infty \).

\subsection{Definitions of the conditional infimum and supremum}
\label{Definitions_of_the_conditional_infimum_and_supremum}

We now give formal definitions of the conditional infimum and supremum
with respect to a correspondence between two sets. 
In optimization, one is more familiar with set-valued mappings 
\cite[Chapter~5]{Rockafellar-Wets:1998} than with correspondences,
though the two notions are essentially equivalent.
We favor the notion of correspondence because, regarding conditional infimum,
we obtain a nicer formula with the composition of correspondences
than with the composition of set-valued mappings  (see
Footnote~\ref{ft:tower_property}). 

\subsubsubsection{Recalls on correspondences}

We recall that a correspondence~$\correspondence$ between two sets~$\DEPART$
and~$\ARRIVEE$ is a subset \( \correspondence \subset \DEPART\times\ARRIVEE \).
We denote \( \depart\correspondence\arrivee \iff 
\np{\depart,\arrivee} \in \correspondence \).
A \emph{foreset} of a correspondence~$\correspondence$ is
any set of the form \( \correspondence\arrivee = 
\bset{\depart \in \DEPART}{ \depart\correspondence\arrivee } \),
where \( \arrivee \in \ARRIVEE \),
or, by extension, of the form
\( \correspondence\Arrivee = \bset{\depart \in \DEPART}%
{ \exists \arrivee\in\Arrivee, \, \depart\correspondence\arrivee } \),
where \( \Arrivee \subset \ARRIVEE \).
An \emph{afterset} of a correspondence~$\correspondence$ is
any set of the form \( \depart\correspondence=
\bset{\arrivee \in \ARRIVEE}{ \depart\correspondence\arrivee } \),
where \( \depart \in \DEPART \),
or, by extension, of the form
\( \Depart\correspondence= \bset{\arrivee \in \ARRIVEE}%
{ \exists  \depart \in \DEPART \eqsepv \depart\correspondence\arrivee } \),
where \( \Depart \subset \DEPART \).
The \emph{domain} and the \emph{range} of the correspondence~$\correspondence$
are given respectively by 
\( \dom\correspondence = \bset{ \depart \in \DEPART }%
{ \depart\correspondence \not= \emptyset } \) 
and
\( \range\correspondence = \bset{ \arrivee \in \ARRIVEE }%
{ \correspondence\arrivee \not= \emptyset } \).
%
%
We denote by \( \correspondence^{-1} \subset \ARRIVEE\times\DEPART \) the correspondence
between the two sets~\( \ARRIVEE \) and \( \DEPART \) given by
\( \arrivee \correspondence^{-1} \depart \iff
\depart \correspondence \arrivee \).
%
For any pair of correspondences $\correspondence$ on~$\DEPART\times\ARRIVEE$
and $\correspondencebis$ on~$\ARRIVEE\times\ARRIVEEter$, the \emph{composition}
\( \correspondence\correspondencebis \) denotes the correspondence
between the two sets~\( \DEPART \) and \( \ARRIVEEter \) given by, 
for any \( \np{\depart,\arriveeter} \in \DEPART\times\ARRIVEEter \),
\( \depart \np{\correspondence\correspondencebis} \arriveeter
  \iff \exists \arrivee \in \ARRIVEE \) such that 
\( \depart \correspondence \arrivee \) and 
\( \arrivee \correspondencebis \arriveeter \).


\subsubsubsection{Optimization over a subset}

Let $ \fonctiondepart : \DEPART \to \barRR $ be a function.
Like in Probability theory\footnote{%
Even if we draw parallels between optimization and
probability theories, we do not develop the parallelism to its potential full extent
as, for instance, we do not consider the equivalent of a generic probability distribution.
Compared to \cite{DelMoral:1997,Akian-Quadrat-Viot:1998,Akian:1999,Barron-Suprema-Jensen:2003} which consider Maslov
measures and densities --- that is, an analog of probability measures --- 
we could say that, in this paper, we only focus on the theory of the conditional infimum/supremum
for the analog of the uniform probability (see the introduction of~\cite{Akian:1999}).
}
where one starts by defining the conditional
probability \wrt\ a subset of the sample space,
we define 
\begin{equation}
  \InfCond{\fonctiondepart}{\Depart} =
  \inf_{\depart \in \Depart} \fonctiondepart\np{\depart}
  \eqsepv
    \SupCond{\fonctiondepart}{\Depart} =
    \sup_{\depart \in \Depart} \fonctiondepart\np{\depart}
    \eqsepv \forall \Depart \subset \DEPART
  \eqfinp
\label{eq:subset_conditional_infimum}  
\end{equation}
To complete the link with optimization under constraint, we also define
\begin{equation}
  \argminInfCond{\fonctiondepart}{\Depart} =
  \argmin_{\depart \in \Depart} \fonctiondepart\np{\depart}
  \eqsepv
    \argmaxSupCond{\fonctiondepart}{\Depart} =
    \argmax_{\depart \in \Depart} \fonctiondepart\np{\depart}
        \eqsepv \forall \Depart \subset \DEPART
  \eqfinp
\label{eq:subset_conditional_infimum_argmin}  
\end{equation}

\subsubsubsection{Definition of conditional infimum \wrt\ a correspondence} 

In the existing definitions of the conditional infimum of a function in the literature
\cite{DelMoral:1997,Akian-Quadrat-Viot:1998,Barron-Suprema-Jensen:2003},
both the original function and its conditional infimum are defined on a measurable space
equipped with a unitary Maslov measure. 
By contrast, our new definition (below) of the conditional infimum of a function
does not require a measurable space (nor a Maslov measure) but a
correspondence between two sets, a source set and a target set; what is more, for a function whose domain
is the source set, its conditional infimum is defined on the target set. 

\begin{definition}
  Let $ \fonctiondepart : \DEPART \to \barRR $ be a function
  and $\correspondence$ be a correspondence between the sets~$\DEPART$ and~$\ARRIVEE$.
  We define the \emph{conditional infimum} 
  of the function~$\fonctiondepart$ with respect to the correspondence~$\correspondence$ 
  (and resp. the \emph{conditional supremum})
  as the functions \( \InfCond{\fonctiondepart}{\correspondence} : \ARRIVEE \to \barRR \)
  (and resp. \( \SupCond{\fonctiondepart}{\correspondence} : \ARRIVEE \to \barRR \))
  given by
  \begin{subequations}
    \begin{align}
      \InfCond{\fonctiondepart}{\correspondence} : 
      \ARRIVEE \to \barRR \eqsepv
      &
      \InfCond{\fonctiondepart}{\correspondence}\np{\arrivee} 
      = 
      \InfCond{\fonctiondepart}{\correspondence\arrivee} 
                                    \eqsepv \forall \arrivee \in \ARRIVEE
                                    \eqfinv
                                    \label{eq:correspondence_conditional_infimum}
      \\
      \SupCond{\fonctiondepart}{\correspondence} : 
      \ARRIVEE \to \barRR \eqsepv
      &
    \SupCond{\fonctiondepart}{\correspondence}\np{\arrivee} 
     = 
\SupCond{\fonctiondepart}{\correspondence\arrivee} 
                                    \eqsepv \forall \arrivee \in \ARRIVEE
                                    \eqfinv
    \end{align}
    \label{eq:correspondence_conditional_infimum_and_supremum}
  \end{subequations}
where we have used the notation~\eqref{eq:subset_conditional_infimum}.   
  \label{de:conditional_infimum} 
\end{definition}

We adopt the conventions\footnote{%
  Such conventions arise naturally as the mapping 
  \( \Depart \in 2^\DEPART \mapsto \inf_{\Depart} \fonctiondepart =
  \inf_{\depart \in \Depart} \fonctiondepart\np{\depart} \) is nonincreasing
  and as the mapping 
  \( \Depart \in 2^\DEPART \mapsto \sup_{\Depart} \fonctiondepart =
  \inf_{\depart \in \Depart} \fonctiondepart\np{\depart} \) is nondecreasing.
  However, one has to be careful because 
  \( \inf_{\Depart} \fonctiondepart \leq \sup_{\Depart} \fonctiondepart \) if
  \( \Depart \neq \emptyset \), but 
  \( +\infty = \inf_{\emptyset} \fonctiondepart > \sup_{\emptyset } \fonctiondepart = -\infty \).
} that 
\cite[p.~1]{Rockafellar-Wets:1998} 
\begin{equation}
  \inf_{\emptyset } \fonctiondepart = 
  \inf_{\depart \in \emptyset } \fonctiondepart\np{\depart} = +\infty 
  \mtext{ and } 
  \sup_{\emptyset } \fonctiondepart =
  \sup_{\depart \in \emptyset } \fonctiondepart\np{\depart} = -\infty \eqfinp
  \label{eq:correspondence_convention}
\end{equation}
As a consequence of~\eqref{eq:correspondence_conditional_infimum_and_supremum}
and~\eqref{eq:correspondence_convention}, the conditional infimum takes the value~$ +\infty $
(and the conditional supremum takes the value~$ -\infty $) outside 
\( \range\correspondence  \), that is, 
\begin{subequations}
  \begin{align}
    \InfCond{\fonctiondepart}{\correspondence}\np{\arrivee} 
    &= 
      +\infty 
      \eqsepv \forall \arrivee \not\in \range\correspondence     
      \eqfinv
    \\
    \SupCond{\fonctiondepart}{\correspondence}\np{\arrivee} 
    &= 
      -\infty 
      \eqsepv \forall \arrivee \not\in \range\correspondence
      \eqfinp
  \end{align}
\end{subequations}
Recall that the \emph{effective domain} of a function 
\( \fonctionarrivee : \ARRIVEE \to \barRR \) is 
\( \dom{\fonctionarrivee}=
\defset{ \arrivee \in \ARRIVEE }{\fonctionarrivee\np{\arrivee} <+\infty}\). 
Therefore, regarding the effective domain, we have the inclusion\footnote{%
  To the left hand side of the inclusion, the notation $\dom$ refers to the 
  effective domain of a \emph{function}, whereas to the right hand side, the notation
$\dom$ refers to the domain of a \emph{correspondence}.}
\begin{equation}
  \dom\bp{\ConditionalInfimum{\correspondence}{\fonctiondepart}} \subset
  \range\correspondence = \dom \correspondence^{-1}
  \eqfinp
  \label{eq:dom_ConditionalInfimum_subset_range}
\end{equation}

All properties about the conditional infimum are easily carried
to the conditional supremum (and conversely) because
\begin{equation}
  -\SupCond{\fonctiondepart}{\correspondence} 
  =  \InfCond{-\fonctiondepart}{\correspondence}
  \eqsepv
  -\InfCond{\fonctiondepart}{\correspondence} 
  = \SupCond{-\fonctiondepart}{\correspondence} \eqfinp
  \label{eq:correspondence_conditional_supremum_infimum}
\end{equation}
In the sequel, we will favour the conditional infimum
as we are interested in applications to minimization problems. 

\subsubsubsection{Example}

  Let $\DirectionRelation \subset \RR^d \times \RR^d $ 
  be the binary relation (hence, correspondence) given by
  \( \primal \DirectionRelation \primal' \iff 
\exists \lambda\in\RR\setminus\na{0} \eqsepv \primal=\lambda\primal' \).
Thus, $\DirectionRelation$ is the equivalence relation on~$\RR^d$
whose classes are~$\na{0}$ and the (unoriented) directions of~$\RR^d$. 
  Let \( \EuclidianNorm{\cdot} \) be the Euclidian norm on~$\RR^d$, 
  $\matrice$ be a matrix with $d$~rows and $p$ columns,
  and $\vecteur \in \RR^p$ be a vector. 
 If we set \( \fonctionprimal\np{\primal}=
 \EuclidianNorm{\matrice\primal-\vecteur}^2 \),
 for \( \primal\in\RR^d \), an easy computation leads to 
\begin{equation*}
\bp{\forall \primal\in\RR^d} \qquad 
  \InfCond{\fonctionprimal}{\DirectionRelation}\np{\primal}=
  \begin{cases}
    \EuclidianNorm{\vecteur}^2 & \mtext{ if } \matrice\primal=0
\eqfinv
\\
  \EuclidianNorm{\vecteur}^2 
        - \frac{ \proscal{\matrice\primal}{\vecteur}^2 }{ \EuclidianNorm{\matrice\primal}^2 }
& \mtext{ if } \matrice\primal \neq 0
\eqfinp 
  \end{cases}
 \end{equation*}

\subsubsubsection{Conditional infimum \wrt\ a 
correspondence induced by a set-valued mapping
\( \Theta : \DEPART \rightrightarrows \ARRIVEE \)}

Let \( \Theta : \DEPART \rightrightarrows \ARRIVEE \) be a set-valued mapping,
that is, \( \Theta : \DEPART \to  2^\ARRIVEE \).
We define the \emph{graph} of~$\Theta$ by
\begin{equation}
  \graph_{ \Theta } = \bset{ \np{\depart,\arrivee} \in \DEPART \times \ARRIVEE }%
  { \arrivee \in \Theta\np{\depart} } \subset \DEPART \times \ARRIVEE 
  \eqfinp
  \label{eq:graph_set-valued_mapping}
\end{equation}
As the graph~$\graph_{ \Theta }$ defines a correspondence between the two sets~\(
\DEPART \) and \( \ARRIVEE \), we introduce 
specific definitions and notations. 
For any function $ \fonctiondepart : \DEPART \to \barRR $, 
we define the \emph{conditional infimum} 
\( \InfCond{\fonctiondepart}{\Theta} : \ARRIVEE \to \barRR \)
of the function~$\fonctiondepart$ with respect to the set-valued mapping~$\Theta$ 
(and the \emph{conditional supremum}
\( \SupCond{\fonctiondepart}{\Theta} : \ARRIVEE \to \barRR \))
by
\begin{subequations}
  \label{eq:ConditionalInfimum_set-valued_mapping}
  \begin{align}
    \InfCond{\fonctiondepart}{ \Theta }\np{\arrivee}
    & = \InfCond{\fonctiondepart}{ \graph_{ \Theta } }\np{\arrivee}
      \eqsepv \forall \arrivee \in \ARRIVEE
      \eqfinv 
    \\
    \SupCond{\fonctiondepart}{ \Theta }\np{\arrivee}
    &= \SupCond{\fonctiondepart}{ \graph_{ \Theta } }\np{\arrivee}
      \eqsepv \forall \arrivee \in \ARRIVEE
      \eqfinp
  \end{align}
\end{subequations}


\subsubsubsection{Conditional infimum \wrt\ a 
correspondence induced by a set-valued mapping 
\( \Phi : \ARRIVEE \rightrightarrows \DEPART \) (the other way round)}

Let \( \Phi : \ARRIVEE \rightrightarrows \DEPART \) be a set-valued mapping.
Beware that, to the difference of the set-valued mapping  
\( \Theta : \DEPART \rightrightarrows \ARRIVEE \) above,
the source set is~$\ARRIVEE$ and the target set is~$\DEPART$.
We consider such set-valued mappings to make the connection with how they are
used in optimization to handle constraints (see
Footnote~\ref{ft:set-valued_mappings_optimization}). 
Regarding the graph of~$\Phi$ in~\eqref{eq:graph_set-valued_mapping}, we have 
\begin{subequations}
  \begin{align}
  \graph_{ \Phi } 
&= 
\bset{ \np{\arrivee,\depart} \in \ARRIVEE \times \DEPART }%
  { \depart \in \Phi\np{\arrivee} } \subset \ARRIVEE \times \DEPART 
                     \eqfinv
\intertext{and, defining \( \Converse{\Phi} : \DEPART \rightrightarrows \ARRIVEE \) 
by \( \Converse{\Phi}\np{\depart} =\nset{\arrivee \in \ARRIVEE }%
   { \depart \in \Phi\np{\arrivee} } \), we get that}
   \npConverse{ \graph_{ \Phi } }
&= 
\bset{ \np{\depart,\arrivee} \in \DEPART \times \ARRIVEE }%
   { \depart \in \Phi\np{\arrivee} } 
    = \graph_{ \Converse{\Phi} }
\subset \DEPART \times \ARRIVEE 
  \eqfinp    
  \end{align}
\end{subequations}
For any function $ \fonctiondepart : \DEPART \to \barRR $, 
we have the properties\footnote{%
  Set-valued mappings are used in optimization because they offer a handy way to
denote constraints as in the left hand side expressions
in~\eqref{eq:ConditionalInfimum_set-valued_mapping_Converse}.
We will explain in Footnote~\ref{ft:tower_property} why we have chosen to favor correspondences
rather than set-valued mappings.
\label{ft:set-valued_mappings_optimization}}
\begin{subequations}
  \begin{align}
\inf_{ \depart \in \Phi\np{\arrivee} }\fonctiondepart\np{\depart}
&= 
\InfCond{\fonctiondepart}{ \graph_{ \Phi }^{-1} }\np{\arrivee}
= \InfCond{\fonctiondepart}{ \Converse{\Phi} }\np{\arrivee}
      \eqsepv \forall \arrivee \in \ARRIVEE
            \eqfinv 
    \\
\sup_{ \depart \in \Phi\np{\arrivee} }\fonctiondepart\np{\depart}
&= 
\SupCond{\fonctiondepart}{ \graph_{ \Phi }^{-1} }\np{\arrivee}
= \SupCond{\fonctiondepart}{ \Converse{\Phi} }\np{\arrivee}
      \eqsepv \forall \arrivee \in \ARRIVEE
      \eqfinv
  \end{align}
\label{eq:ConditionalInfimum_set-valued_mapping_Converse}
\end{subequations}
where we have used the notations in~\eqref{eq:ConditionalInfimum_set-valued_mapping}.

\subsubsubsection{Conditional infimum \wrt\ a 
correspondence induced by a mapping \( \theta : \DEPART \to \ARRIVEE \)}

Let \( \theta : \DEPART \to \ARRIVEE \) be a  mapping.
The graph of~$\theta$ in~\eqref{eq:graph_set-valued_mapping} is now
\begin{equation}
  \graph_{ \theta } = \bset{ \np{\depart,\arrivee} \in \DEPART \times \ARRIVEE }%
  { \theta\np{\depart} = \arrivee } \subset \DEPART \times \ARRIVEE 
  \eqfinp
  \label{eq:graph}
\end{equation}
For any function $ \fonctiondepart : \DEPART \to \barRR $, 
Equations~\eqref{eq:ConditionalInfimum_set-valued_mapping} give
(with the mapping~$\theta$ identified with the set-valued mapping $\depart
\mapsto \ba{\theta(\depart)}$)
\begin{subequations}
  \begin{align}
    \InfCond{\fonctiondepart}{ \theta }\np{\arrivee}
    & = \InfCond{\fonctiondepart}{ \graph_{ \theta } }\np{\arrivee}
    =  
    \inf_{\depart \in \theta^{-1}\np{\na{\arrivee}}} \fonctiondepart\np{\depart} 
      \eqsepv \forall \arrivee \in \ARRIVEE
      \eqfinv 
      \label{eq:ConditionalInfimum}
    \\
    \SupCond{\fonctiondepart}{ \theta }\np{\arrivee}
    &= \SupCond{\fonctiondepart}{ \graph_{ \theta } }\np{\arrivee}
    =  
    \sup_{\depart \in \theta^{-1}\np{\na{\arrivee}}} \fonctiondepart\np{\depart} 
      \eqsepv \forall \arrivee \in \ARRIVEE
      \eqfinp
  \end{align}
\end{subequations}
As \( \theta : \DEPART \to \ARRIVEE \) is a mapping,
it induces a set-valued mapping
\( \Converse{\theta} :  \ARRIVEE \rightrightarrows \DEPART \).
Using Equations~\eqref{eq:ConditionalInfimum_set-valued_mapping},
for any function $\fonctionarrivee:\ARRIVEE \to \barRR$, we have the following properties:
\begin{subequations}
  \begin{align}
    \InfCond{\fonctionarrivee}{ \Converse{\theta} }\np{\depart}
=    \InfCond{\fonctionarrivee}{\graph_{ \theta }^{-1} }\np{\depart}
    &=
      \inf_{\arrivee = \theta\np{\depart} }
      \fonctionarrivee\np{\arrivee}
      =
      \bp{\fonctionarrivee \circ \theta }\np{\depart}
      \eqsepv \forall \depart \in \DEPART
      \eqfinv 
      \label{eq:mapping_composition_and_ConditionalInfimum}
    \\
    \SupCond{\fonctionarrivee}{ \Converse{\theta} }\np{\depart}
=   \SupCond{\fonctionarrivee}{\graph_{\theta}^{-1} }\np{\depart}
    &=  
      \sup_{\arrivee = \theta\np{\depart} }
      \fonctionarrivee\np{\arrivee}
      =
      \bp{\fonctionarrivee \circ \theta }\np{\depart}
      \eqsepv \forall \depart \in \DEPART
      \eqfinp
  \end{align}
\end{subequations}

\subsection{Examples}
\label{Examples_of_the_conditional_infimum_and_supremum}

\subsubsubsection{Examples with characteristic functions}

For any subset \( \Uncertain \subset \UNCERTAIN \) of a set~\( \UNCERTAIN \),
$\delta_{\Uncertain} : \UNCERTAIN \to \barRR $ denotes the \emph{characteristic function} of the
set~$\Uncertain$: 
\( \delta_{\Uncertain}\np{\uncertain} = 0 \) if \( \uncertain \in \Uncertain \),
and \( \delta_{\Uncertain}\np{\uncertain} = +\infty \) 
if \( \uncertain \not\in \Uncertain \).

For any correspondence~$\correspondence$ between the sets~$\DEPART$ and $\ARRIVEE$, and
for any \( \depart \in \DEPART \) and any subset \( \Depart \subset \DEPART \),
we have that,
\begin{subequations}
  \begin{align}
    \InfCond{ \delta_{ \{ \depart \} } }{\correspondence} 
    &=  
      \delta_{ \depart\correspondence }   
      \eqsepv &
                \InfCond{ \delta_{ \Depart } }{\correspondence} 
                =  
                \delta_{ \Depart\correspondence }   
                \eqfinv 
    \\
    \SupCond{ -\delta_{ \{ \depart \} } }{\correspondence} 
    &=  
      -\delta_{ \depart\correspondence }   
      \eqsepv &
                \SupCond{ -\delta_{ \Depart } }{\correspondence} 
                =  
                \delta_{ -\Depart\correspondence }   
                \eqfinp
  \end{align}
  \label{eq:correspondence_SupCond_delta}
\end{subequations}
As a consequence, the conditional infimum and supremum
with respect to a correspondence characterize this latter,
as the mappings
\( \correspondence \mapsto \InfCond{\cdot}{\correspondence} \)
and \( \correspondence \mapsto \SupCond{\cdot}{\correspondence} \)
are injective.

\subsubsubsection{Examples with rectangular correspondences}

For any two subsets \( \Depart \subset \DEPART \) and
\( \Arrivee \subset \ARRIVEE \), we define the \emph{rectangle correspondence}
\( \Depart{\times}\Arrivee \).
Then, for any function \( \fonctiondepart : \DEPART \to \barRR \), we have that 
\begin{subequations}
  \begin{align*}
    \bInfCond{\fonctiondepart}{\Depart{\times}\Arrivee}\np{\arrivee}
    &=
      \begin{cases}
        \inf_{\depart \in \Depart} \fonctiondepart\np{\depart} 
        & \text{if } \arrivee\in \Arrivee 
        \\ 
        +\infty 
        &\text{if } \arrivee \not\in \Arrivee 
      \end{cases}
          \mtext{ that is, }
   \bInfCond{\fonctiondepart}{\Depart{\times}\Arrivee}
    =
      \InfCond{\fonctiondepart}{\Depart} \UppPlus \delta_{ \Arrivee } 
      \eqfinv          
    \\
    \bSupCond{\fonctiondepart}{\Depart{\times}\Arrivee}\np{\arrivee}
    &=
      \begin{cases}
        \sup_{\depart \in \Depart} \fonctiondepart\np{\depart} 
        & \text{if } \arrivee\in \Arrivee 
        \\ 
        -\infty 
        &\text{if } \arrivee \not\in \Arrivee 
      \end{cases} 
          \mtext{ that is, }
   \bSupCond{\fonctiondepart}{\Depart{\times}\Arrivee}
    =
      \SupCond{\fonctiondepart}{\Depart} \LowPlus \np{ -\delta_{ \Arrivee } }
      \eqfinp 
  \end{align*}
\end{subequations}


\subsubsubsection{Marginalization operations}

Let $\DEPART$ and~$\ARRIVEE$ be two sets,
\( \Delta_\ARRIVEE \) be the diagonal of \( \ARRIVEE^2 \),
and \( \DEPART{\times}\Delta_\ARRIVEE \) be the correspondence between the
sets~$\DEPART{\times}\ARRIVEE$ and~$\ARRIVEE$ given by 
\(   \np{\depart,\arrivee} \bp{\DEPART{\times}\Delta_\ARRIVEE} \arrivee'
\iff \arrivee=\arrivee' \).
The foresets of the correspondence~\( \DEPART{\times}\Delta_\ARRIVEE \)
satisfy \( \bp{\DEPART{\times}\Delta_\ARRIVEE} \arrivee = \DEPART{\times}\na{\arrivee} \),
for any \( \arrivee \in \ARRIVEE \).
The following conditional infimum and supremum of a 
function~$\FonctionDepartArrivee : \DEPART{\times}\ARRIVEE \to \barRR$ 
with respect to the correspondence~\( \DEPART{\times}\Delta_\ARRIVEE \)
provide the \emph{marginalization operations}:
\begin{subequations}
  \begin{align}
                                    \InfCond{\FonctionDepartArrivee}{\DEPART{\times}\Delta_\ARRIVEE}\np{\arrivee} 
    &= 
      \inf_{\depart \in \DEPART} \FonctionDepartArrivee\np{\depart,\arrivee} 
      \eqsepv \forall \arrivee \in \ARRIVEE
      \eqfinv
    \\
                                    \SupCond{\FonctionDepartArrivee}{\DEPART{\times}\Delta_\ARRIVEE}\np{\arrivee} 
    &= \sup_{\depart \in \DEPART} \FonctionDepartArrivee\np{\depart,\arrivee} 
      \eqsepv \forall \arrivee \in \ARRIVEE
      \eqfinp
  \end{align}
\end{subequations}


\subsection{Properties of the conditional infimum and supremum}
\label{Properties_of_the_conditional_infimum_and_supremum}

We expose properties of the conditional infimum and supremum.
We recall that 
the \emph{strict epigraph} of a function $ \fonctiondepart : \DEPART \to \barRR $ is
defined by the subset 
\begin{equation}
  \StrictEpigraph\fonctiondepart =
  \defset{ (\depart,t) \in \DEPART\times\RR }{%
    \fonctiondepart(\depart) < t }
 \subset  \DEPART\times\RR 
 \eqfinv
 \label{eq:StrictEpigraph}
\end{equation}
hence \( \StrictEpigraph\fonctiondepart \)
can be understood as a correspondence between~$\DEPART$ and $\barRR $.

\begin{proposition}
  \quad
  \begin{enumerate}
  \item 
    Strict epigraph:\\ 
    for any function $ \fonctiondepart : \DEPART \to \barRR $
    and for any correspondence~$\correspondence$ on~$\DEPART\times\ARRIVEE$, 
    we have that
    \begin{equation}
      \StrictEpigraph\InfCond{\fonctiondepart}{\correspondence}=
      \Converse{\correspondence} \bp{\StrictEpigraph\fonctiondepart}
      \eqfinv
      \label{eq:correspondence_conditional_infimum_strict_epigraph} 
    \end{equation}
    where, on the right hand side, \( \StrictEpigraph\fonctiondepart \) is
    understood as a correspondence between~$\DEPART$ and $\barRR $,
    and \(  \Converse{\correspondence} \bp{\StrictEpigraph\fonctiondepart}
    \)
    as a composition of two correspondences, hence as a subset of \(
    \ARRIVEE\times\RR \).
  \item
    Linearity and sublinearity \wrt\ min-plus, $\wedge$ and $\vee$ operations:\\ 
    for any correspondence~$\correspondence$ on~$\DEPART\times\ARRIVEE$, 
    we have that
    \begin{subequations}
      \begin{align}
        \intertext{$\bullet$ for any family \( \np{\fonctiondepart_i}_{i \in I} \) of functions 
        $ \fonctiondepart_i : \DEPART \to \barRR $,}
        \InfCond{ \Bwedge_{i \in I} \fonctiondepart_i }{\correspondence} 
        &=
          \Bwedge_{i \in I} \InfCond{ \fonctiondepart_i }{\correspondence} 
          \eqfinv 
          \label{eq:correspondence_conditional_infimum_properties_wedge}
        \\ 
        \Bvee_{i \in I} \InfCond{ \fonctiondepart_i }{\correspondence} 
        &\leq
          \InfCond{ \Bvee_{i \in I} \fonctiondepart_i }{\correspondence} 
          \eqfinv 
          \intertext{$\bullet$ for any function $ \fonctiondepart : \DEPART \to \barRR $ 
          and \( r \in \barRR \),}     
        \InfCond{\fonctiondepart \UppPlus r}{\correspondence}
        &= 
           \InfCond{\fonctiondepart}{\correspondence} \UppPlus r 
           \eqfinv 
           \intertext{$\bullet$ for any  functions $ \fonctiondepart : \DEPART \to \barRR $ 
           and $ \fonctiondepartbis : \DEPART \to \barRR $,}
           \InfCond{\fonctiondepart}{\correspondence} 
           \UppPlus \InfCond{\fonctiondepartbis}{\correspondence} 
        &\leq
          \InfCond{\fonctiondepart \UppPlus \fonctiondepartbis}{\correspondence}
          \eqfinp
      \end{align}
    \end{subequations}
  \item 
    Monotony with respect to functions:\\
    for any correspondence~$\correspondence$ on~$\DEPART\times\ARRIVEE$, 
    we have that
    \begin{subequations}
      \begin{align}
        \intertext{$\bullet$ for any functions $ \fonctiondepart : \DEPART \to \barRR $ and $ \fonctiondepartbis : \DEPART \to \barRR $,}
        \fonctiondepart \leq \fonctiondepartbis \implies
        \InfCond{\fonctiondepart}{\correspondence} 
        &\leq 
          \InfCond{\fonctiondepartbis}{\correspondence} 
          \eqfinv \label{eq:correspondence_conditional_infimum_properties_monotony}
        \\
        \InfCond{\fonctiondepart}{\DEPART}
        \leq 
        \InfCond{\fonctiondepart}{\correspondence\ARRIVEE} 
        &\leq \InfCond{\fonctiondepart}{\correspondence}
          (\arrivee) \eqsepv \forall \arrivee \in \ARRIVEE
          \eqfinv 
          %
          \intertext{$\bullet$ for any function $ \fonctiondepart : \DEPART \to \barRR $ 
          and for any nondecreasing function \( \varphi : \barRR \to \barRR \),}
        \varphi \circ \InfCond{\fonctiondepart }{\correspondence} 
        &\leq
          \InfCond{\varphi \circ \fonctiondepart}{\correspondence}
          \eqfinp
      \end{align}
    \end{subequations}
  \item 
    Two correspondences on~$\DEPART\times\ARRIVEE$:\\
    for any pair $\correspondence$, $\correspondencebis$ of correspondences
    on~$\DEPART\times\ARRIVEE$ and for any 
    function $ \fonctiondepart : \DEPART \to \barRR $,
    we have that
    \begin{subequations}
      \begin{align}
        \bInfCond{f}{\correspondence \cup \correspondencebis}
        &= 
          \InfCond{f}{\correspondence} \bwedge \InfCond{f}{\correspondencebis} 
          \eqfinv 
        \\
        \bInfCond{f}{\correspondence \cap \correspondencebis}
        &\geq 
          \InfCond{f}{\correspondence} \bvee \InfCond{f}{\correspondencebis} 
          \eqfinv 
        \\
        \correspondence \subset \correspondencebis \implies
        \InfCond{f}{\correspondence}
        &\geq \InfCond{f}{\correspondencebis} 
          \eqfinp
          \label{eq:correspondence_conditional_infimum_properties_inclusion} 
      \end{align}
    \end{subequations}
  \item 
    Pushforward property:\footnote{%
      This formula for the conditional infimum has the flavour of the
      change of variable formula under pushforward probability.}\\
    for any correspondence $\correspondence$ on~$\DEPART\times\ARRIVEE$,
    for any subset \( \Arrivee \subset \ARRIVEE \)
    and for any function $ \fonctiondepart : \DEPART \to \barRR $,
    we have that 
    \begin{equation}
      \bInfCond{ \nInfCond{\fonctiondepart}{\correspondence} }{\Arrivee}
      = 
        \bInfCond{\fonctiondepart}{\correspondence\Arrivee}
          \eqfinp
          \label{eq:correspondence_conditional_infimum_properties_pushforward}
    \end{equation}
  \item 
    Tower property:\footnote{%
      This formula for conditional infima has the flavour of the tower property for
      conditional expectations. Had we defined the conditional infimum not \wrt\ a correspondence,
      but \wrt\ a set-valued mapping, the tower property would write, in a
      reverse way, as 
\( \bInfCond{ \nInfCond{\fonctiondepart}{ \Converse{\Phi} } }{ \Converse{\Psi} }
      = 
      \bInfCond{\fonctiondepart}{ \npConverse{ \Psi \circ \Phi} } \),
making appear the composition~$\Psi \circ \Phi$ of two set-valued mappings 
\( \Psi : \ARRIVEEbis \rightrightarrows \ARRIVEE \) 
and 
\( \Phi : \ARRIVEE \rightrightarrows \DEPART \) 
as in \cite[p.~151]{Rockafellar-Wets:1998}.
Indeed, the composition~$\Psi \circ \Phi$ 
satisfies \( \graph_{\Psi \circ \Phi}=\graph_{\Phi}\graph_{\Psi} \), hence
\( \graph_{ \npConverse{ \Psi \circ \Phi} }=
\npConverse{ \graph_{\Psi \circ \Phi} }=
\npConverse{ \graph_{\Phi}\graph_{\Psi} }=
\npConverse{ \graph_{\Psi} } \npConverse{ \graph_{\Phi} } \).
We prefer the formula
\(  \bInfCond{ \nInfCond{\fonctiondepart}{\correspondence} }{\correspondencebis} 
      = 
      \bInfCond{\fonctiondepart}{\correspondence\correspondencebis} \)
to the  formula     
\( \bInfCond{ \nInfCond{\fonctiondepart}{ \Converse{\Phi} } }{ \Converse{\Psi} }
      = 
      \bInfCond{\fonctiondepart}{ \npConverse{ \Psi \circ \Phi} } \). 
\label{ft:tower_property} }\\
    for any pair of correspondences $\correspondence$ on~$\DEPART\times\ARRIVEE$
    and $\correspondencebis$ on~$\ARRIVEE\times\ARRIVEEter$,
    and for any function $ \fonctiondepart : \DEPART \to \barRR $,
    we have that 
    \begin{equation}
      \bInfCond{ \nInfCond{\fonctiondepart}{\correspondence} }{\correspondencebis} 
      = 
      \bInfCond{\fonctiondepart}{\correspondence\correspondencebis} 
      \eqfinp
      \label{eq:tower_property}
    \end{equation}
  \item 
    Right composition with mappings:\\
    for any correspondence~$\correspondence$ on~$\DEPART\times\ARRIVEE$, 
    we have that 

    \noindent $\bullet$ for any function $ \fonctiondepartbis : \DEPARTbis \to \barRR $ 
    and for any mapping \( \theta : \DEPART \to \DEPARTbis \),
    \begin{subequations}
      \begin{align}
        \InfCond{\fonctiondepartbis \circ \theta}{\correspondence} 
        &= 
          \bInfCond{\fonctiondepartbis}{  \graph_{ \theta }^{-1}\correspondence }
          \eqfinv  
          \label{eq:correspondence_conditional_infimum_right_composition_correspondence}
          \intertext{$\bullet$ for any function $ \fonctiondepart : \DEPART \to \barRR $ 
          and for any mapping \( \theta : \ARRIVEEbis \to \ARRIVEE \),}
          \InfCond{\fonctiondepart}{\correspondence}  \circ \theta 
        &= 
          \bInfCond{\fonctiondepart}{  \correspondence\graph_{ \theta }^{-1} }
          \eqfinp 
          \label{eq:correspondence_conditional_infimum_correspondence_right_composition}
      \end{align}
    \end{subequations}
  \item 
    Joint conditional infimum and supremum:\\
    for any correspondence~$\correspondence$ on~$\DEPART\times\ARRIVEE$, 
    and for any functions $ \fonctiondepart : \DEPART \to \barRR $ 
    and $ \fonctiondepartbis : \DEPART \to \barRR $,
    we have that
    \begin{subequations}
      \begin{align}
        \InfCond{\fonctiondepart \UppPlus \fonctiondepartbis}{\correspondence}  
        & \leq 
          \InfCond{\fonctiondepart}{\correspondence} 
          \UppPlus \SupCond{\fonctiondepartbis}{\correspondence} 
          \eqfinv 
        \\
        \SupCond{\fonctiondepart \LowPlus \fonctiondepartbis}{\correspondence} 
        & \geq 
          \SupCond{\fonctiondepart}{\correspondence} 
          \LowPlus \InfCond{\fonctiondepartbis}{\correspondence} 
          \eqfinp 
          \label{eq:correspondence_conditional_infimum_properties_sum_geq}
      \end{align}
    \end{subequations}
  \end{enumerate}
\end{proposition}

\begin{proof} 
  Most of the claims are straightforward 
  consequences of the Definition~\ref{de:conditional_infimum} of the
  conditional infimum, and are left to the reader.

  \noindent $\bullet$
  We prove~\eqref{eq:correspondence_conditional_infimum_strict_epigraph}:
  \begin{align*}
    \np{\arrivee,t} \in \StrictEpigraph\InfCond{\fonctiondepart}{\correspondence}
    & \iff
\InfCond{\fonctiondepart}{\correspondence}\np{\arrivee} < t      
\tag{by definition~\eqref{eq:StrictEpigraph} of the strict epigraph}
\\
& \iff      
\exists \depart \in \correspondence\arrivee 
\eqsepv \fonctiondepart\np{\depart} < t 
\tag{by definition~\eqref{eq:correspondence_conditional_infimum} of the
     conditional infimum \( \InfCond{\fonctiondepart}{\correspondence} \)}
\\
& \iff
\exists \depart \in \DEPART \eqsepv
\arrivee \Converse{\correspondence} \depart 
\text{ and } \np{\depart,t} \in \StrictEpigraph\fonctiondepart
      \tag{by definition~\eqref{eq:StrictEpigraph} of the strict epigraph}
    \\
& \iff
\exists \depart \in \DEPART \eqsepv
\arrivee \Converse{\correspondence} \depart 
\text{ and } \depart ~\bp{{\StrictEpigraph\fonctiondepart}}~ t 
\\
& \iff
\np{\arrivee,t} \in \Converse{\correspondence} \bp{\StrictEpigraph\fonctiondepart}
     \eqfinp
     \tag{by definition of the composition of correspondences}
  \end{align*}

  \noindent $\bullet$
  We prove~\eqref{eq:correspondence_conditional_infimum_right_composition_correspondence}
  as follows.
  For any correspondence~$\correspondence$ on~$\DEPART\times\ARRIVEE$, 
  any function $ \fonctiondepartbis : \DEPARTbis \to \barRR $,
  any mapping \( \theta : \DEPART \to \DEPARTbis \)
  and any \( \arrivee\in\ARRIVEE \), we have that 
  \begin{align*}
    \nInfCond{\fonctiondepartbis \circ \theta}{\correspondence}
    &=
      \bInfCond{\nInfCond{\fonctiondepartbis}{\graph_{\theta}^{-1}}}{\correspondence}
      \tag{as $\fonctiondepartbis \circ \theta=\InfCond{\fonctiondepartbis}{\graph_{\theta}^{-1}}$
      by~\eqref{eq:mapping_composition_and_ConditionalInfimum}}
    \\
    &=\bInfCond{\fonctiondepartbis}{  \graph_{ \theta }^{-1}\correspondence }\eqfinp
      \tag{by the tower property~\eqref{eq:tower_property}}
  \end{align*}

  \noindent $\bullet$ We prove~\eqref{eq:correspondence_conditional_infimum_correspondence_right_composition} as follows.
  For any correspondence~$\correspondence$ on~$\DEPART\times\ARRIVEE$, 
  any function $ \fonctiondepart : \DEPART \to \barRR $,
  any mapping \( \theta : \ARRIVEEbis \to \ARRIVEE \)
  and any \( \arriveebis \in \ARRIVEEbis \), we have that 
  \begin{align*}
    { \nInfCond{\fonctiondepart}{\correspondence}  \circ \theta }
    &=
      \bInfCond{\nInfCond{\fonctiondepart}{\correspondence}}{{\graph_{\theta}^{-1}}}
      \tag{by~\eqref{eq:mapping_composition_and_ConditionalInfimum}}
    \\
    &=
      \InfCond{\fonctiondepart}{ \correspondence \graph_{ \theta }^{-1} }
      \tag{by the tower property~\eqref{eq:tower_property}}
      \eqfinp 
  \end{align*}

  \noindent $\bullet$ We prove~\eqref{eq:tower_property} as follows.
  For any pair of correspondences $\correspondence$ on~$\DEPART\times\ARRIVEE$
  and $\correspondencebis$ on~$\ARRIVEE\times\ARRIVEEter$,
  any function $ \fonctiondepart : \DEPART \to \barRR $
  and any \( \arriveebis \in \ARRIVEEbis \), we have that 
  \begin{align*}
    \bInfCond{ \nInfCond{\fonctiondepart}{\correspondencebis} }{\correspondence}\np{\arriveebis}
    &=
      \inf_{ \arrivee \in \correspondence\arriveebis }
      \InfCond{\fonctiondepart}{\correspondencebis}\np{\arrivee} 
      \tag{by definition~\eqref{eq:correspondence_conditional_infimum} of the 
      conditional infimum}
    \\
    &=
      \inf_{ \arrivee \in \correspondence\arriveebis }
      \inf_{ \depart \in \correspondencebis\arrivee }
      \fonctiondepart\np{\depart}
      \tag{by definition~\eqref{eq:correspondence_conditional_infimum} of the 
      conditional infimum}
    \\
    &=
      \inf_{ \depart \in \correspondencebis\arrivee, \arrivee \in \correspondence\arriveebis }
      \fonctiondepart\np{\depart}
    \\
    &=
      \inf_{ \depart \in \correspondencebis\correspondence\arriveebis }
      \fonctiondepart\np{\depart}
      \tag{by definition of the composition of two correspondences}
    \\
    &=
      \bInfCond{\fonctiondepart}{\correspondence\correspondencebis} 
      \np{\arriveebis}
      \tag{by definition~\eqref{eq:correspondence_conditional_infimum} of the 
      conditional infimum}
      \eqfinp
  \end{align*}
  \medskip

  This ends the proof.
\end{proof}

\section{Applications of the conditional infimum to minimization problems}
\label{Applications_of_the_conditional_infimum_to_minimization_problems}

With the conditional infimum, we now establish equalities 
and inequalities between two minimization problems,
an original problem on the set~$\UNCERTAIN$
and another one on the set~$\PRIMAL$, where the sets
$\UNCERTAIN$ and $\PRIMAL$ are possibly different
(in particular, $\PRIMAL$ might be a vector space, whereas $\UNCERTAIN$ is not). 
As we deal with optimization problems, we will often resort to the more telling
usage \( \inf_{\uncertain \in \Uncertain} \fonctionuncertain\np{\uncertain} \)
or \( \min_{\uncertain \in \Uncertain} \fonctionuncertain\np{\uncertain} \),
rather than \( \InfCond{\fonctionuncertain}{\Uncertain} \) as in~\eqref{eq:subset_conditional_infimum}.   

\begin{proposition}
  We consider two sets $\UNCERTAIN$ and $\PRIMAL$, 
  a correspondence~$\correspondence$ on~$\UNCERTAIN\times\PRIMAL$, 
  and a function \( \fonctionuncertain : \UNCERTAIN \to \barRR \).
  For any subset \( \Primal \subset \PRIMAL \), 
  we have the equality
  \begin{equation}
    \inf_{\uncertain \in \correspondence\Primal} \fonctionuncertain\np{\uncertain} 
    =
    \inf_{\primal \in \Primal}
    \bp{\nInfCond{\fonctionuncertain}{\correspondence}\np{\primal}}
    \eqfinp
    \label{eq:ConditionalInfimum_tower_property_equality_subset}
  \end{equation}  
  For any subset \( \Uncertain \subset \UNCERTAIN \), we have the implications
  \begin{subequations}
    \begin{align}
   \Uncertain \subset \domain\correspondence
& \implies
   \inf_{\uncertain \in \Uncertain} \fonctionuncertain\np{\uncertain} 
    \geq
    \inf_{\primal \in \Uncertain\correspondence} 
    \bp{\nInfCond{\fonctionuncertain}{\correspondence}\np{\primal}}
    \eqfinv
    \label{eq:ConditionalInfimum_tower_property_ineq}
\\
\correspondence\correspondence^{-1}\Uncertain \subset 
      \Uncertain \subset \domain\correspondence
   & \implies
      \inf_{\uncertain \in \Uncertain} \fonctionuncertain\np{\uncertain} 
      =
      \inf_{\primal \in \Uncertain\correspondence} 
      \bp{\nInfCond{\fonctionuncertain}{\correspondence}\np{\primal}}
      \eqfinp
      \label{eq:ConditionalInfimum_tower_property}
    \end{align}
\end{subequations}

  %
          %
  \label{pr:inf_ConditionalInfimum_inf_original}
\end{proposition}

\begin{proof}
  The equality~\eqref{eq:ConditionalInfimum_tower_property_equality_subset}
  is proved as follows: 
  \begin{align*}
    \inf_{\uncertain \in \correspondence\Primal} \fonctionuncertain\np{\uncertain} 
    &=
      \InfCond{\fonctionuncertain}{\correspondence\Primal}
     \tag{by definition~\eqref{eq:subset_conditional_infimum}}
    \\
    &=
\bInfCond{ \nInfCond{\fonctionuncertain}{\correspondence} }{\Primal}
\tag{by the pushforward property~\eqref{eq:correspondence_conditional_infimum_properties_pushforward}}
      \\
    &=
    \inf_{\primal \in \Primal}
    \bp{\InfCond{\fonctionuncertain}{\correspondence}\np{\primal}}
      \eqfinp
           \tag{by definition~\eqref{eq:subset_conditional_infimum}}
  \end{align*}

We suppose that \( \Uncertain \subset \domain\correspondence\) 
and we prove the right hand side inequality
in~\eqref{eq:ConditionalInfimum_tower_property_ineq}.
  First, we prove that   $\Uncertain \subset
  \correspondence\correspondence^{-1}\Uncertain$. 
Indeed, if $\uncertain \in \Uncertain$ we have that
  $\uncertain \in \domain\correspondence$ as
\( \Uncertain \subset \domain\correspondence\).
Therefore, there exists $\primal \in \PRIMAL $ such that 
$\uncertain \correspondence \primal$ or, equivalently, 
that $\primal \correspondence^{-1} \uncertain$.
Now, $\uncertain \correspondence \primal$ 
and $\primal \correspondence^{-1} \uncertain$
imply that  $\uncertain  \correspondence \correspondence^{-1} \uncertain$ and thus $\uncertain
  \in \correspondence \correspondence^{-1} \Uncertain$. Second, we obtain that
  \begin{align*}
    \inf_{\uncertain \in \Uncertain} \fonctionuncertain\np{\uncertain}
    &\ge 
      \inf_{\uncertain \in \correspondence \correspondence^{-1} \Uncertain} \fonctionuncertain\np{\uncertain}
      \tag{since \( \Uncertain \subset \correspondence\correspondence^{-1}\Uncertain \)}
    \\
    &=
    \inf_{\primal \in\correspondence^{-1}\Uncertain}
      \bp{\InfCond{\fonctionuncertain}{\correspondence}\np{\primal}}
      \eqfinp
      \tag{by Equation~\eqref{eq:ConditionalInfimum_tower_property_equality_subset} with $\Primal=\correspondence^{-1}\Uncertain$}
  \end{align*}

When \( \correspondence\correspondence^{-1}\Uncertain \subset 
      \Uncertain \subset \domain\correspondence \), 
the right hand side equality
in~\eqref{eq:ConditionalInfimum_tower_property} comes from the fact that the
inequality above is an equality using that \( \correspondence\correspondence^{-1}\Uncertain \subset 
      \Uncertain \).
  
  \medskip

  This ends the proof. 
\end{proof}

With the conditional infimum, we now state sufficient conditions to relate the 
optimal solutions of two minimization problems. 

\begin{proposition}
  We consider a function \( \fonctionuncertain : \UNCERTAIN \to \barRR \), 
  a subset \( \Uncertain \subset \UNCERTAIN \) 
  and the minimization problem 
  \begin{equation}
    \min_{ \uncertain\in\Uncertain } \fonctionuncertain\np{\uncertain} 
    \eqfinp  
    \label{eq:minimization_problem_Uncertain}
  \end{equation}
  Assume that there exists 
  \begin{subequations}
    \begin{enumerate}
    \item 
      a set~$\PRIMAL$,
      a correspondence~$\correspondence$ on~$\UNCERTAIN\times\PRIMAL$, 
      and a function 
      \( \fonctionprimal : \PRIMAL \to \barRR \) such that 
      \begin{equation}
        \fonctionprimal\np{\primal} \leq 
        \InfCond{\fonctionuncertain}{\correspondence}\np{\primal}
        \eqsepv \forall \primal \in \PRIMAL
        \eqfinv
        \label{eq:argmin_ConditionalInfimum_argmin_original_fonctionprimal}
      \end{equation}
      a subset 
      \( \Primal  \subset \PRIMAL \) such that 
      \begin{equation}
        \Uncertain \subset \correspondence \Primal  
        \eqfinv
        \label{eq:argmin_ConditionalInfimum_argmin_original_subsets}
      \end{equation}
      and an optimal solution~\( \primal\opt \in \PRIMAL \) to the 
      auxiliary minimization problem
      \( \min_{ \primal \in \Primal } \fonctionprimal\np{\primal} \),
      that is, 
      \begin{equation}
        \primal\opt
        \in 
        \argmin_{ {\primal \in \Primal} } \fonctionprimal\np{\primal}
        \eqfinv
        \label{eq:argmin_ConditionalInfimum_argmin_original_primal}
      \end{equation}
    \item 
      an element \( \uncertain\opt \in \UNCERTAIN \) such that 
      \begin{align}
        \fonctionuncertain\np{\uncertain\opt}
        &=
          \fonctionprimal\np{\primal\opt}
          \eqfinv
          \label{eq:argmin_ConditionalInfimum_argmin_original_=}
        \\
        \uncertain\opt
        &\in 
          \Uncertain
          \eqfinp
          \label{eq:argmin_ConditionalInfimum_argmin_original_in_Uncertain=}
      \end{align}
    \end{enumerate}
  \end{subequations}
  Then, $\uncertain\opt$ is an optimal solution to the original
  minimization problem~\eqref{eq:minimization_problem_Uncertain}, 
  that is, 
  \begin{equation}
    \uncertain\opt \in 
    \argmin_{ \uncertain \in\Uncertain } \fonctionuncertain\np{\uncertain}
    \eqfinp
    \label{eq:argmin_ConditionalInfimum=argmin_original}
  \end{equation}
  \label{pr:argmin_ConditionalInfimum_argmin_original}
\end{proposition}

\begin{proof}
  The equality~\eqref{eq:argmin_ConditionalInfimum=argmin_original}
  between solutions of minimization problems follows from
  \begin{align*}
    \fonctionuncertain\np{\uncertain\opt}
    &=
      \fonctionprimal\np{\primal\opt}
      \tag{by assumption~\eqref{eq:argmin_ConditionalInfimum_argmin_original_=}}
    \\
    &=
      \min_{ \primal \in \Primal } \fonctionprimal\np{\primal}
      \tag{by assumption~\eqref{eq:argmin_ConditionalInfimum_argmin_original_primal}}
    \\
    &\leq
      \inf_{ \primal \in \Primal } \InfCond{\fonctionuncertain}{\correspondence}\np{\primal}
      \tag{because \( \fonctionprimal \leq 
      \InfCond{\fonctionuncertain}{\correspondence}\)
      by assumption~\eqref{eq:argmin_ConditionalInfimum_argmin_original_fonctionprimal}}
    \\
    &=
      \inf_{ \uncertain \in \correspondence\Primal } \fonctionuncertain\np{\uncertain}
      \tag{by the equality~\eqref{eq:ConditionalInfimum_tower_property_equality_subset}}
    \\
    & \leq 
      \inf_{ \uncertain \in \Uncertain} \fonctionuncertain\np{\uncertain}
      \eqfinp
      \tag{because \( \correspondence \Primal \supset \Uncertain \)
      by assumption~\eqref{eq:argmin_ConditionalInfimum_argmin_original_subsets}}
  \end{align*}
  As \( \uncertain\opt \in \Uncertain \)
  by assumption~\eqref{eq:argmin_ConditionalInfimum_argmin_original_in_Uncertain=},
  this ends the proof.
\end{proof}

\section{Detecting hidden convexity using the conditional infimum}
\label{Detecting_hidden_convexity_in_optimization_problems_using_conditional_infimum}

In~\S\ref{A_sufficient_condition_for_hidden_convexity_in_optimization_problems},
we provide a sufficient condition for hidden convexity in minimization
problems.
Then, in~\S\ref{Hidden_convexity_in_quadratic_optimization_problems},
we show how our result applies to quadratic optimization problems.

\subsection{A sufficient condition for hidden convexity in minimization
  problems}
\label{A_sufficient_condition_for_hidden_convexity_in_optimization_problems}

We propose a formal definition of ``hidden convexity'' in minimization
problems, using the notation~\eqref{eq:subset_conditional_infimum}--\eqref{eq:subset_conditional_infimum_argmin}. 

\begin{definition}
  \begin{subequations}
We consider a set~$\UNCERTAIN$, 
a function \( \fonctionuncertain : \UNCERTAIN \to \barRR \)
and a subset \( \Uncertain \subset \UNCERTAIN \).
We say that the minimization problem
\( \inf_{\uncertain \in \Uncertain} \fonctionuncertain\np{\uncertain}
= \InfCond{\fonctionuncertain}{\Uncertain} \)
displays \emph{hidden convexity}
if there exists a vector space~$\PRIMAL$, 
a convex function \( \fonctionprimal: \PRIMAL \to \barRR \) 
and a convex subset \( \Convex \subset \PRIMAL \) such that 
\begin{equation}
  \inf_{\uncertain \in \Uncertain} \fonctionuncertain\np{\uncertain}
=    \InfCond{\fonctionuncertain}{\Uncertain}
= \InfCond{\fonctiondepart}{\Convex}
= \inf_{\primal\in\Convex}\fonctiondepart\np{\primal}
\eqfinp
    \label{eq:hidden_convexity}
  \end{equation}
  Moreover, the minimization problem
\( \min_{\uncertain \in \Uncertain} \fonctionuncertain\np{\uncertain} \) 
  is said to display
  \emph{strong hidden convexity} if, in addition to~\eqref{eq:hidden_convexity},
  $\argmin \nsetc{\fonctiondepart}{\Convex}\not=\emptyset$
  and there exists a set-valued mapping
  $\gamma: \Convex \rightrightarrows \Uncertain$ such that
  \begin{equation}
    \gamma\bp{\argmin \nsetc{\fonctiondepart}{\Convex}} \subset \argmin \nsetc{\fonctionuncertain}{\Uncertain}
      \eqfinp
    \end{equation}
  \end{subequations}
    \label{de:hidden_convexity}
\end{definition}

We state a sufficient condition for hidden convexity in minimization
problems, using the conditional infimum.

\begin{proposition}
  Let $\PRIMAL$ be a vector space,
  $\UNCERTAIN$ be a set and 
  $\correspondence \subset \UNCERTAIN \times \PRIMAL$ 
  be a correspondence between the sets~$\UNCERTAIN$ and~$\PRIMAL$.
  Let \( \fonctionuncertain : \UNCERTAIN \to \barRR \) be a function,
  and \( \Convex \subset \PRIMAL \) be a convex subset such that 
  the function~\( \InfCond{\fonctionuncertain}{\correspondence}:
  \PRIMAL \to \barRR \) is convex on~$\Convex$. 
  Then, the minimization problem
  \( \inf_{\uncertain \in \correspondence\Convex} 
  \fonctionuncertain\np{\uncertain} \) displays hidden convexity
  as in~\eqref{eq:hidden_convexity}, with \( \fonctionprimal=
  \InfCond{\fonctionuncertain}{\correspondence} \):
  \begin{equation}
    \inf_{\uncertain \in \correspondence\Convex} 
    \fonctionuncertain\np{\uncertain} 
    =
    \inf_{\primal \in \Convex} 
    \bp{\InfCond{\fonctionuncertain}{\correspondence}} \np{\primal} 
    \eqfinp
    \label{eq:optimization_hidden_convexity}
  \end{equation}
  Moreover, if there exists an optimal solution~\( \primal\opt \in \PRIMAL \) 
  to the auxiliary convex minimization problem
  \( \min_{\primal \in \Convex} 
  \bp{\InfCond{\fonctionuncertain}{\correspondence}} \np{\primal} \), 
  that is, if 
  \begin{subequations}
    \begin{align}
      \primal\opt 
      &\in 
        \argmin_{\primal \in \Convex} 
        \bp{\InfCond{\fonctionuncertain}{\correspondence}}\np{\primal} 
        \eqfinv
        \label{eq:optimization_hidden_convexity_argmin_ConditionalInfimum_argmin_original_primal}
      \intertext{and if there exists an optimal solution~\( \uncertain\opt \) 
      to the minimization problem 
      \( min_{\uncertain \in \correspondence \primal\opt }
      \fonctionuncertain\np{\uncertain} \) ---  
      which is the original minimization problem but with 
      stronger constraint \( \uncertain \in \correspondence \primal\opt \) 
      instead of \( \uncertain \in \correspondence\Convex \) ---
      that is, if}
      \uncertain\opt
      &\in 
        \argmin_{\uncertain \in \correspondence \primal\opt } \fonctionuncertain\np{\uncertain}
        \label{eq:optimization_hidden_convexity_argmin_ConditionalInfimum_uncertain_opt}
        \eqfinv
        \intertext{then $\uncertain\opt$ is an optimal solution to the original
        minimization problem \( \min_{\uncertain \in \correspondence\Convex} 
        \fonctionuncertain\np{\uncertain} \), that is,}
        \uncertain\opt
      &\in 
        \argmin_{\uncertain \in \correspondence\Convex} 
        \fonctionuncertain\np{\uncertain} 
        \eqfinp 
    \end{align}  
  \end{subequations}
  \label{pr:optimization_hidden_convexity}
\end{proposition}

\begin{proof}
  The equality~\eqref{eq:optimization_hidden_convexity}
  is a straightforward application of
  the equality~\eqref{eq:ConditionalInfimum_tower_property_equality_subset}
  with \( \Primal=\Convex \), in Proposition~\ref{pr:inf_ConditionalInfimum_inf_original}.

  The second part regarding the $\argmin$ is an application of
  Proposition~\ref{pr:argmin_ConditionalInfimum_argmin_original}
  whose assumptions 
  \eqref{eq:argmin_ConditionalInfimum_argmin_original_fonctionprimal}, 
  \eqref{eq:argmin_ConditionalInfimum_argmin_original_subsets}, 
  \eqref{eq:argmin_ConditionalInfimum_argmin_original_primal}, 
  \eqref{eq:argmin_ConditionalInfimum_argmin_original_=}, 
  \eqref{eq:argmin_ConditionalInfimum_argmin_original_in_Uncertain=}
  are satisfied as follows. 

  Equation~\eqref{eq:argmin_ConditionalInfimum_argmin_original_fonctionprimal}
  is satisfied by taking the function \( \fonctionprimal =
  \InfCond{\fonctionuncertain}{\correspondence} \).
  Equation~\eqref{eq:argmin_ConditionalInfimum_argmin_original_subsets}
  is satisfied by taking the subsets \( \Primal=\Convex \) and 
  \( \Uncertain=\correspondence\Convex \).
  Equation~\eqref{eq:argmin_ConditionalInfimum_argmin_original_primal}
  is exactly
  Equation~\eqref{eq:optimization_hidden_convexity_argmin_ConditionalInfimum_argmin_original_primal}
  in the assumptions as \( \fonctionprimal =
  \InfCond{\fonctionuncertain}{\correspondence} \)
  and \( \Primal=\Convex \).
  Equation~\eqref{eq:argmin_ConditionalInfimum_argmin_original_=}
  holds true because 
  \begin{align*}
    \fonctionuncertain\np{\uncertain\opt}
    &=
      \min_{\uncertain \in \correspondence \primal\opt }
      \fonctionuncertain\np{\uncertain}
      \tag{by the
      assumption~\eqref{eq:optimization_hidden_convexity_argmin_ConditionalInfimum_uncertain_opt}}
    \\
    &=
      \InfCond{\fonctionuncertain}{\correspondence}\np{\primal\opt}
      \tag{by definition~\eqref{eq:correspondence_conditional_infimum} of the 
      conditional infimum~\( \InfCond{\fonctionuncertain}{\correspondence} \)}
    \\
    &=
      \fonctionprimal\np{\primal\opt}
      \eqfinp
      \tag{because \( \fonctionprimal =
      \InfCond{\fonctionuncertain}{\correspondence} \)}
  \end{align*}
  Equation~\eqref{eq:argmin_ConditionalInfimum_argmin_original_in_Uncertain=}
  is satisfied because 
  \( \uncertain\opt \in \correspondence \primal\opt \)
  by~\eqref{eq:optimization_hidden_convexity_argmin_ConditionalInfimum_uncertain_opt},
  where \( \primal\opt \in \Convex \)
  by~\eqref{eq:optimization_hidden_convexity_argmin_ConditionalInfimum_argmin_original_primal},
  so that \( \uncertain\opt \in \correspondence\Convex=\Uncertain \).
\end{proof}

In the next~\S\ref{Hidden_convexity_in_quadratic_optimization_problems},
we will use the following version of
Proposition~\ref{pr:optimization_hidden_convexity}
where the correspondence~$\correspondence$ is induced by a mapping.

\begin{corollary}
  Let $\PRIMAL$ be a vector space,
  $\UNCERTAIN$ be a set and 
\( \theta: \UNCERTAIN \to \PRIMAL \) be a mapping.
  Let \( \fonctionuncertain : \UNCERTAIN \to \barRR \) be a function,
  and \( \Convex \subset \PRIMAL \) be a convex subset such that 
  the function~\( \InfCond{\fonctionuncertain}{\theta}:
  \PRIMAL \to \barRR \) is convex on~$\Convex$. 
  Then, the minimization problem
  \( \inf_{\theta\np{\uncertain} \in \Convex} 
  \fonctionuncertain\np{\uncertain} \) displays hidden convexity
  as in~\eqref{eq:hidden_convexity}, with \( \fonctionprimal=
  \InfCond{\fonctionuncertain}{\theta} \):
  \begin{equation}
    \inf_{\theta\np{\uncertain} \in \Convex} 
    \fonctionuncertain\np{\uncertain} 
    =
    \inf_{\primal \in \Convex} 
    \bp{\InfCond{\fonctionuncertain}{\theta}} \np{\primal} 
    \eqfinp
    \label{eq:optimization_hidden_convexity_2}
  \end{equation}
  Moreover, if there exists an optimal solution~\( \primal\opt \in \PRIMAL \) 
  to the auxiliary convex minimization problem
  \( \min_{\primal \in \Convex} 
  \bp{\InfCond{\fonctionuncertain}{\theta}} \np{\primal} \), 
  that is, if 
  \begin{subequations}
    \begin{align}
      \primal\opt 
      &\in 
        \argmin_{\primal \in \Convex} 
        \bp{\InfCond{\fonctionuncertain}{\theta}}\np{\primal} 
        \eqfinv
        \label{eq:optimization_hidden_convexity_argmin_ConditionalInfimum_argmin_original_primal_2}
      \intertext{and if there exists an optimal solution~\( \uncertain\opt \) 
      to the minimization problem 
      \( min_{\theta\np{\uncertain}=\primal\opt }
      \fonctionuncertain\np{\uncertain} \) ---  
      which is the original minimization problem but with the
      stronger constraint \( \theta\np{\uncertain}=\primal\opt \) 
      instead of \( \theta\np{\uncertain} \in \Convex \) ---
      that is, if}
      \uncertain\opt
      &\in 
        \argmin_{\theta\np{\uncertain}=\primal\opt } \fonctionuncertain\np{\uncertain}
        \label{eq:optimization_hidden_convexity_argmin_ConditionalInfimum_uncertain_opt_2}
        \eqfinv
        \intertext{then $\uncertain\opt$ is an optimal solution to the original
        minimization problem \( \min_{\theta\np{\uncertain} \in \Convex} 
        \fonctionuncertain\np{\uncertain} \), that is,}
        \uncertain\opt
      &\in 
        \argmin_{\theta\np{\uncertain} \in \Convex} 
        \fonctionuncertain\np{\uncertain} 
        \eqfinp 
    \end{align}  
  \end{subequations}
  \label{cor:optimization_hidden_convexity}
\end{corollary}

\subsection{Hidden convexity in the quadratic case}
\label{Hidden_convexity_in_quadratic_optimization_problems}

We study hidden convexity both for functions and for minimization
problems in the quadratic case.
Let $d \in \NN^*$ be a positive integer.
We define  the \emph{square mapping}
\( \SquareMapping: \RR^d \to \RR^d \) by
\begin{equation}
  \SquareMapping\np{\uncertain}=\SquareMapping\np{\uncertain_1,\ldots,\uncertain_d}
  = \np{\uncertain_1^2,\ldots,\uncertain_d^2} 
  \eqsepv \forall \uncertain \in \RR^d 
  \eqfinp 
  \label{eq:SquareMapping}
\end{equation}
      %

\subsubsection{Hidden convexity in linear-quadratic functions}

We provide necessary and sufficient conditions under which 
the conditional infimum of a linear-quadratic function, \wrt\ 
the square mapping~\eqref{eq:SquareMapping}, is convex.
We deduce a sufficient condition for hidden convexity
of a linear-quadratic function \wrt\ to the square mapping.

\begin{proposition}
  Let $d \in \NN^*$ be a positive integer, $\vecteur \in \RR^d$ be a vector, 
  and $\matrice$ be a $d\times d$ symmetric matrix.
  Let the linear-quadratic
  function \( \LinearQuadraticFonctionuncertain : \RR^d \to \barRR \) be given
  by (where~$'$ denotes transposition)
  \begin{subequations}
    \begin{equation}
      \LinearQuadraticFonctionuncertain\np{\uncertain}=
      \uncertain'\matrice\uncertain + \vecteur'\uncertain 
      \eqsepv \forall  \uncertain \in \RR^d 
      \eqfinp 
      \label{eq:LinearQuadraticFonctionuncertain}
    \end{equation}
    Then, the function \( \fonctionprimal=\ConditionalInfimum{\SquareMapping}{\LinearQuadraticFonctionuncertain} : \RR^d
    \to \barRR \), defined in~\eqref{eq:ConditionalInfimum} by
    \begin{equation}
      \fonctionprimal\np{\primal}=
      \inf\bset{ \uncertain'\matrice\uncertain + \vecteur'\uncertain }%
      { \uncertain_1^2=\primal_1,\ldots,\uncertain_d^2=\primal_d } 
      \eqsepv \forall \primal=\np{\primal_1,\ldots,\primal_d} \in \RR^d 
      \eqfinv
      \label{eq:LinearQuadraticFonctionuncertain_wrt_SquareMapping}
    \end{equation}
    is convex if and only if
    \begin{equation}
      \exists \varepsilon=\np{\varepsilon_1,\ldots,\varepsilon_d} \in \na{-1,1}^d \text{ such that }
      \begin{cases}
        \varepsilon_i\vecteur_i \leq 0 \eqsepv & \forall i=1,\ldots,d 
        \eqfinv
        \\
        \text{and} &
        \\
        \varepsilon_i\varepsilon_j\matrice_{ij}  \leq 0 \eqsepv 
        & 
        \forall i,j=1,\ldots,d \eqsepv i\neq j
        \eqfinp 
      \end{cases}
      \label{eq:varepsilon}
    \end{equation}
    In that case, the function~\( \fonctionprimal=
    \ConditionalInfimum{\SquareMapping}{\LinearQuadraticFonctionuncertain}\) 
    in~\eqref{eq:LinearQuadraticFonctionuncertain_wrt_SquareMapping} 
    is proper convex lsc 
    with effective domain \( \domain\fonctionprimal =\RR_+^d\), 
    and has the expression
    \begin{equation}
      \fonctionprimal\np{\primal_1,\ldots,\primal_d} =
      \begin{cases}
        +\infty & \text{ if } \np{\primal_1,\ldots,\primal_d} \not\in \RR_+^d
        \eqfinv
        \\
        \displaystyle 
        \sum_{i=1}^d \matrice_{ii} \primal_i 
        -
        \sum_{i\neq j} \module{\matrice_{ij}}\sqrt{\primal_i \primal_j}
        -
        \sum_{i=1}^d \module{\vecteur_i} \sqrt{\primal_i }
        & \text{ if } \np{\primal_1,\ldots,\primal_d} \in \RR_+^d 
        \eqfinp
      \end{cases}
      \label{eq:ConditionalInfimum_quadratic-example}
    \end{equation}
    As a consequence, if~\eqref{eq:varepsilon} holds true,
    the linear-quadratic function \( \LinearQuadraticFonctionuncertain : \RR^d \to \barRR \) 
    displays hidden convexity  (see Footnote~\ref{ft:hidden_convexity_function})
    with respect to the square mapping~\( \SquareMapping \)
    as we have that 
    \( \LinearQuadraticFonctionuncertain = \fonctionprimal \circ \SquareMapping \),
    where the function~\( \fonctionprimal \) is convex. 
  \end{subequations}
  \label{pr:Conditional_infimum_of_a_quadratic_function_knowing_squares}
\end{proposition}

\begin{proof}
  \begin{subequations}
    The proof is in three steps.
    \medskip

    \noindent$\bullet$ 
    First, we obtain different expressions of the function~\( \fonctionprimal =
    \ConditionalInfimum{\SquareMapping}{\LinearQuadraticFonctionuncertain}\) 
    in~\eqref{eq:LinearQuadraticFonctionuncertain_wrt_SquareMapping}.
    As \( \SquareMapping\np{\RR^d}= \RR_+^d\) by
    definition~\eqref{eq:SquareMapping} of the square mapping, we have
    \(  \bp{\ConditionalInfimum{\SquareMapping}{\LinearQuadraticFonctionuncertain}}\np{\primal}
    =+\infty \) for any \( \primal \not\in \RR_+^d\), by~\eqref{eq:dom_ConditionalInfimum_subset_range}.
    Then, we have, for any \( \primal=\np{\primal_1,\ldots,\primal_d}
    \in \RR_+^d \), 
    \begin{align}
      \fonctionprimal\np{\primal_1,\ldots,\primal_d} 
      &=   \bp{\ConditionalInfimum{
        \SquareMapping}{\LinearQuadraticFonctionuncertain}}
        \np{\primal_1,\ldots,\primal_d}
        \nonumber
      \\
      &=
        \inf\bset{\uncertain'\matrice\uncertain + \vecteur'\uncertain }{%
        \uncertain\in\RR^d \eqsepv  
        \np{\uncertain_1^2,\ldots,\uncertain_d^2} = \np{\primal_1,\ldots,\primal_d} }
        \nonumber
        \intertext{by definition~\eqref{eq:ConditionalInfimum} of the conditional
        infimum \wrt\ a mapping, and by definition~\eqref{eq:SquareMapping} of the square mapping}
      &=
        \inf\Bset{ \sum_{i=1}^d \matrice_{ii} \uncertain_i^2 +
        \sum_{i\neq j} \matrice_{ij} \uncertain_i\uncertain_j 
        + \sum_{i=1}^d \vecteur_i \uncertain_i }{%
        \uncertain_1^2=\primal_1,\ldots,\uncertain_d^2=\primal_d}
        \nonumber
      \\
      &=
        \sum_{i=1}^d \matrice_{ii} \primal_i +
        \inf\Bset{ 
        \sum_{i\neq j} \matrice_{ij} \uncertain_i\uncertain_j 
        + \sum_{i=1}^d \vecteur_i \uncertain_i }{%
        \uncertain_1=\pm\sqrt{\primal_1},\ldots,\uncertain_d=\pm\sqrt{\primal_d} }
        \nonumber
      \\
      &=
        \sum_{i=1}^d \matrice_{ii} \primal_i +
        \min\Bset{ 
        \sum_{i\neq j} \matrice_{ij}\varepsilon'_i\varepsilon'_j \sqrt{\primal_i\primal_j} 
        + \sum_{i=1}^d \vecteur_i\varepsilon'_i \sqrt{\primal_i} }{%
        \varepsilon' \in \na{-1,1}^d \} }
        \label{eq:ConditionalInfimum_quadratic-example_varepsilon'_inproof}
      \\
      & \geq 
        \sum_{i=1}^d \matrice_{ii} \primal_i +
        \sum_{i\neq j} \module{\matrice_{ij}} \np{-\sqrt{\primal_i \primal_j} }
        + \sum_{i=1}^d \module{\vecteur_i} \np{-\sqrt{\primal_i }}
        \eqfinv
        \label{eq:ConditionalInfimum_quadratic-example_inproof}
    \end{align}
    where we recognize, in this last expression~\eqref{eq:ConditionalInfimum_quadratic-example_inproof},
    the expression~\eqref{eq:ConditionalInfimum_quadratic-example} of the 
    function~\( \fonctionprimal  \).
    \medskip

    \noindent$\bullet$ 
    Second, we suppose that~\eqref{eq:varepsilon} holds true.
    Then, it is easy to check that \( \varepsilon \in \na{-1,1}^d  \) given
    by~\eqref{eq:varepsilon} 
    provides an equality in the inequality
    between~\eqref{eq:ConditionalInfimum_quadratic-example_varepsilon'_inproof}
    and~\eqref{eq:ConditionalInfimum_quadratic-example_inproof}.
    Thus, we get that the function 
    \( \ConditionalInfimum{\SquareMapping}{\LinearQuadraticFonctionuncertain} \)
    is the function \( \fonctionprimal \) 
    given by~\eqref{eq:ConditionalInfimum_quadratic-example}.
    Now, it is easily checked (by computing the Hessian) that the functions 
    \( \np{\uncertain_i,\uncertain_j} \in \RR_+^2 \mapsto \np{-\sqrt{\primal_i
        \primal_j} } \) are convex, for all $i\neq j$.
    Therefore, it is easily deduced that the function \( \fonctionprimal : \RR^d \to \barRR \) 
    in~\eqref{eq:ConditionalInfimum_quadratic-example} is convex lsc 
    with effective domain \( \domain\fonctionprimal =\RR_+^d\), 
    hence is proper convex lsc. 
    \medskip

    \noindent$\bullet$ 
    Third, we suppose that the function 
    \( \fonctionprimal=\ConditionalInfimum{\SquareMapping}{\LinearQuadraticFonctionuncertain} : \RR^d \to \barRR \)
    is convex.

    For any \( \varepsilon \in \na{-1,1}^d \), 
    the following subset \( \Primal_{\varepsilon} \) of \( ]0,+\infty[^d \) is
    closed (as easily follows from its second expression below)
    \begin{align*}
      \Primal_{\varepsilon}
      =&
        \Bset{ \primal\in ]0,+\infty[^d }{ \varepsilon \in \argmin\defset{ 
        \sum_{i\neq j} \matrice_{ij}\varepsilon'_i\varepsilon'_j \sqrt{\primal_i\primal_j} 
        + \sum_{i=1}^d \vecteur_i\varepsilon'_i \sqrt{\primal_i} }{%
        \varepsilon' \in \na{-1,1}^d \} } }
      \\
      =&
         \bigg\{ \primal\in ]0,+\infty[^d
         \; \bigg\vert \;
        \sum_{i\neq j} \matrice_{ij}\varepsilon_i\varepsilon_j \sqrt{\primal_i\primal_j} 
        + \sum_{i=1}^d \vecteur_i\varepsilon_i \sqrt{\primal_i} 
      \\
       & \qquad\qquad\qquad\qquad \leq 
         \sum_{i\neq j} \matrice_{ij}\varepsilon'_i\varepsilon'_j \sqrt{\primal_i\primal_j} 
         + \sum_{i=1}^d \vecteur_i\varepsilon'_i \sqrt{\primal_i} \eqsepv
         \forall \varepsilon' \in \na{-1,1}^d \bigg\}
         \eqfinp 
    \end{align*}
    We are going to show that one of the subsets \( \Primal_{\varepsilon} \), when 
    \( \varepsilon \in \na{-1,1}^d \), has nonempty interior.
    As \( \bigcup_{\varepsilon' \in \na{-1,1}^d } \Primal_{\varepsilon'}  = ]0,+\infty[^d
    \), there is at least one subset \( L \subset \na{-1,1}^d \) such that 
    \( \bigcup_{\varepsilon' \in L} \Primal_{\varepsilon'}  = ]0,+\infty[^d \)
    and the subset~$L$ has the smallest possible cardinal.
    If \( \cardinal{L}=1 \), then there is one \( \varepsilon \in \na{-1,1}^d \) such that \(
    \Primal_{\varepsilon} = ]0,+\infty[^d\), and this \( \Primal_{\varepsilon} \) obviously has nonempty interior.
    If \( \cardinal{L} \geq 2 \), then
\( \bigcup_{\varepsilon' \in \na{-1,1}^d } \Primal_{\varepsilon'}  = ]0,+\infty[^d \)
implies that, for any \( \varepsilon \in L \), 
    we have that \( \emptyset \subsetneq \Bp{ \bigcup_{\varepsilon' \in L\setminus \na{\varepsilon}}
      \Primal_{\varepsilon'} }^c \subset \Primal_{\varepsilon} \).
    Therefore, the subset \( \Primal_{\varepsilon} \) has nonempty interior since 
    it contains the nonempty set  \( \Bp{ \bigcup_{\varepsilon' \in L\setminus \na{\varepsilon}}
      \Primal_{\varepsilon'} }^c \), which is open as the complementary set of 
    a finite union of closed subsets.

    As a consequence, there is one \( \varepsilon \in \na{-1,1}^d \) 
    and there is a ball~$B$ in~\( ]0,+\infty[^d \) such that 
    \[
      \fonctionprimal\np{\primal_1,\ldots,\primal_d}
      =
      \sum_{i\neq j} \matrice_{ij}\varepsilon_i\varepsilon_j \sqrt{\primal_i\primal_j} 
      + \sum_{i=1}^d \vecteur_i\varepsilon_i \sqrt{\primal_i} 
      \eqsepv \forall \primal \in B
      \eqfinp
    \]
    As the fonction \( \fonctionprimal \) is convex, 
    so is the function \( k: B \ni \primal \mapsto \sum_{i\neq j} \matrice_{ij}\varepsilon_i\varepsilon_j \sqrt{\primal_i\primal_j} 
    + \sum_{i=1}^d \vecteur_i\varepsilon_i \sqrt{\primal_i} \), and so are 
    the restrictions \( \primal_i \mapsto k\np{0,\ldots,0,\primal_i,0,\ldots,0} \)
    and \( \np{\primal_i,\primal_j} \mapsto
    k\np{0,\ldots,0,\primal_i,0,\ldots,0,\primal_j,0,\ldots,0} \)
    for any \( i \neq j \). 
    We conclude readily that \( \varepsilon \in \na{-1,1}^d \) satifies~\eqref{eq:varepsilon}.
    \medskip

    This ends the proof.
  \end{subequations}
\end{proof}

\subsubsection{Hidden convexity in linear-quadratic minimization problems}

Now, we provide sufficient conditions under which the minimization of 
a linear-quadratic function, under constraints given by the square
mapping~\eqref{eq:SquareMapping}, 
displays hidden convexity.

\begin{proposition}
  Let $d \in \NN^*$ be a positive integer
  and \( \Convex \subset \RR_+^d \) be a convex subset.
  Let $\vecteur \in \RR^d$ be a vector, 
  and $\matrice$ be a $d\times d$ symmetric matrix 
  such that~\eqref{eq:varepsilon} holds true.
  Then, the minimization problem
  \( \min_{\uncertain\in\RR^d} \uncertain'\matrice\uncertain + \vecteur'\uncertain \),
  under the constraint that \( \np{\uncertain_1^2,\ldots,\uncertain_d^2} \in \Convex \),
  displays strong hidden convexity, as in Definition~\ref{de:hidden_convexity}.

  \begin{subequations}
  Indeed, we have that 
  \begin{equation}
    \inf\defset{\uncertain'\matrice\uncertain + \vecteur'\uncertain }{%
      \uncertain\in\RR^d \eqsepv  
      \np{\uncertain_1^2,\ldots,\uncertain_d^2} \in \Convex }
    =
    \inf_{ \primal \in \Convex }\fonctionprimal\np{\primal}
    \eqfinv
    \label{eq:ConditionalInfimum_linear-quadratic_problem}
  \end{equation}
  where 
  the function \( \fonctionprimal : \RR^d \to \barRR \)
  is proper convex lsc with effective domain \( \domain\fonctionprimal=\RR_+^d\),
  and is given by~\eqref{eq:ConditionalInfimum_quadratic-example}.
 
  Moreover, regarding argmin, we have the following implication
  \begin{equation}
    \primal\opt \in \argmin_{ \primal \in \Convex }\fonctionprimal\np{\primal}
    \implies 
    \varepsilon \cdot \sqrt{\primal\opt} \in 
    \argmin\defset{\uncertain'\matrice\uncertain + \vecteur'\uncertain }{%
      \uncertain\in\RR^d \eqsepv  
      \np{\uncertain_1^2,\ldots,\uncertain_d^2} \in \Convex }
    \eqfinv
    \label{eq:argmin}
  \end{equation}
  \end{subequations}
  where the vector \( \varepsilon \cdot \sqrt{\primal\opt} \in \RR^d \)
  has components 
  \( \varepsilon_i\sqrt{\primal\opt_i} \), for \( i=1,\ldots,d \)
  and \( \varepsilon \) is given by~\eqref{eq:varepsilon}.
  \label{pr:ConditionalInfimum_quadratic-example}
\end{proposition}

\begin{proof}
  Equation~\eqref{eq:ConditionalInfimum_linear-quadratic_problem} is 
  a straightforward application of 
  Corollary~\ref{cor:optimization_hidden_convexity}.
  Indeed, Equation~\eqref{eq:ConditionalInfimum_linear-quadratic_problem}
  follows from Equation~\eqref{eq:optimization_hidden_convexity}
  with 
  function \( \fonctionuncertain=\LinearQuadraticFonctionuncertain\) given
  by~\eqref{eq:LinearQuadraticFonctionuncertain},
  correspondence \( \correspondence=\graph_{\SquareMapping} \) given
  by the graph of the square mapping~\eqref{eq:SquareMapping},
  and convex subset \( \Convex \subset \RR_+^d \).

  The correspondence~\eqref{eq:argmin} between argmins also follows 
  from Corollary~\ref{cor:optimization_hidden_convexity}, 
  by using the vector
  \( \uncertain\opt = \varepsilon \cdot \sqrt{\primal\opt} \in \RR^d \)
  where \( \varepsilon \) is given by~\eqref{eq:varepsilon}.
  \medskip

  This ends the proof. 
\end{proof}

Our result covers (and extends) the following two cases.

\begin{corollary}
  \begin{subequations}
    For any convex subset \( \Convex \subset \RR_+^d \) and 
    symmetrix matrix \( \Matrix \) such that 
    \( \Matrix_{ij}  \geq 0 \) for all \( i \neq j \),
    the maximization problem
    \begin{equation}
      \max\defset{\uncertain'\Matrix\uncertain }{%
        \uncertain\in\RR^d \eqsepv  
        \np{\uncertain_1^2,\ldots,\uncertain_d^2} \in \Convex } 
    \end{equation}
    is equivalent to the convex minimization problem
    \begin{equation}
      \min\defset{- \sum_{i=1}^d \Matrix_{ii} \primal_i -
        \sum_{i\neq j} \Matrix_{ij} \sqrt{\primal_i \primal_j} }%
      { \np{\primal_1,\ldots,\primal_d} \in \Convex } 
      \eqfinp   
    \end{equation}
  \end{subequations}
\end{corollary}

\begin{proof}
  It suffices to apply
  Proposition~\ref{pr:ConditionalInfimum_quadratic-example}
  with \( \matrice=-\Matrix \), \( \vecteur=0 \) 
  and \( \varepsilon =\np{1,\ldots,1} \).
\end{proof}

The following Corollary extends the result in~\cite[Theorem~7]{BenTal-Ben-Teboulle:1996},
as we do not require the simultaneous diagonalization property\footnote{%
The simultaneous diagonalization property
  \cite[Equation~(3)]{BenTal-Ben-Teboulle:1996} reads as:
  \( \exists \eta\in\RR \) such that \( \matrice + \eta S > 0 \).}.

\begin{corollary}
  \begin{subequations}
    For any scalars \( l \leq u \) 
    and any diagonal matrices \( \matrice \) and~$S$,
    the minimization problem 
    \begin{equation}
      \min\Bset{\uncertain'\matrice\uncertain + \vecteur'\uncertain }{%
        \uncertain\in\RR^d \eqsepv  
        l \leq \uncertain' S \uncertain \leq u }   
    \end{equation}
    is equivalent to the convex minimization problem
    \begin{equation}
      \min\Bset{ \sum_{i=1}^d \matrice_{ii} \primal_i 
        -
        \sum_{i=1}^d \module{\vecteur_i} \sqrt{\primal_i } }%
      { \np{\primal_1,\ldots,\primal_d} \in \RR_+^d \eqsepv
        l \leq \sum_{i=1}^d S_{ii} \primal_i \leq u } 
      \eqfinp   
    \end{equation}
  \end{subequations}
\end{corollary}

\begin{proof}
  It suffices to apply
  Proposition~\ref{pr:ConditionalInfimum_quadratic-example}
  with \( \Convex =\defset{ \np{\primal_1,\ldots,\primal_d} \in \RR_+^d }%
  { l \leq \sum_{i=1}^d S_{ii} \primal_i \leq u } \)
  and \( \varepsilon =- \textrm{sign} \np{\vecteur} \). 
\end{proof}

\section{Conclusion}
\label{Conclusion}

Detecting hidden convexity is one of the tools to address nonconvex
minimization problems.
In this paper, we have contributed to this research program by
giving a formal definition of hidden convexity in a minimization problem,
and by putting forward the notion of conditional infimum. 
Building upon a well-known parallelism between optimization and probability
theories, we have established a list of properties of the conditional infimum,
among which a tower formula, relevant for minimization problems. 
Thus equipped, we have provided sufficient conditions for 
  hidden convexity in nonconvex optimization problems, and we have illustrated our results on
  nonconvex quadratic minimization problems.
  We finish this conclusion by pointing out perspectives for using
  the conditional infimum in other contexts, namely in relation to
  the so-called S-procedure, to couplings and conjugacies,
  and to lower bound convex programs.
  \medskip
  
  The conditional infimum appears in the so-called S-procedure
  (see the survey paper~\cite{Polik-Terlaky:2007}),
itself related to hidden convexity (see~\cite{Ben-Tal-den-Hertog-Laurent:2011})
  as follows. 
Let $ \fonctiondepart_0, \fonctiondepart_1, \ldots, \fonctiondepart_p : \DEPART
\to \RR $ be functions, and consider the statements
\begin{subequations}
  \begin{align}
    (I)     \qquad &
                     \bp{ \fonctiondepart_i\np{\depart} \geq 0 \eqsepv \forall i=1,\ldots,p }
                     \implies
                     \fonctiondepart_0\np{\depart} \geq 0 
                     \eqfinv
                     \label{eq:(I)}
    \\
    (C) \qquad &
                 \exists\, \alpha_1 \geq 0, \ldots, \alpha_p \geq 0
                 \mtext{ such that }\fonctiondepart_0 - \sum_{i=1}^p  \alpha_i\fonctiondepart_i \geq 0
                 \eqfinp 
                 \label{eq:(C)}
  \end{align}
\end{subequations}
It is obvious that (C) $\implies$ (I). The S-procedure consists in finding
sufficient conditions to ensure that (I) $\implies$ (C), that is, 
conditions such that (I) + conditions $\implies$ (C).
We easily show that 
we can write statement~(I) in term of conditional infimum as 
\begin{equation}
     (I) 
          \iff
\InfCond{\fonctiondepart_0}{ \fonctiondepart_1, \ldots, \fonctiondepart_p
          }\np{\arrivee} \geq 0 
          \eqsepv \forall \arrivee \in \RR_+^p
\eqfinp           
\end{equation}

The conditional infimum is related to couplings and conjugacies. 
One-sided linear couplings are introduced
in~\cite{Chancelier-DeLara:2021_ECAPRA_JCA} as follows:
letting $\UNCERTAIN$ be a set and \( \theta: \UNCERTAIN \to \RR^d \) be a mapping, we define the
coupling~\( \Fenchelcoupling_\theta: \UNCERTAIN \times \RR^d \to \RR \)
by \( \Fenchelcoupling_\theta\couple{\dual}{\uncertain}
=\FenchelCoupling{\theta\np{\uncertain}}{\dual} \),
for any \( \np{\uncertain, \dual} \in \UNCERTAIN\times\RR^d \). 
Then, we show in 
\cite[Proposition~2.5]{Chancelier-DeLara:2021_ECAPRA_JCA}
that the $\Fenchelcoupling_\theta$-Fenchel-Moreau conjugate
\( \SFM{\fonctionuncertain}{\Fenchelcoupling_\theta} \) of a
function \( \fonctionuncertain : \UNCERTAIN \to \barRR \)
can be expressed as the Fenchel conjugate of the
conditional infimum~\( \InfimalPostComposition{\theta}{\fonctionuncertain} \):
    \begin{equation}
      \SFM{\fonctionuncertain}{\Fenchelcoupling_\theta}=
      \LFM{ \bp{\InfimalPostComposition{\theta}{\fonctionuncertain}} }
      \eqfinp
      \label{eq:CAPRA_Fenchel-Moreau_conjugate}
    \end{equation}
    It appears that one-sided linear couplings are also related to
    hidden convexity: in
    \cite[Proposition~2.6]{Chancelier-DeLara:2021_ECAPRA_JCA},
    we show that a function is $\Fenchelcoupling_\theta$-convex 
  if and only if it is the composition of a closed convex function on~$\RR^d$
  with the mapping~$\theta$ (see Footnote~\ref{ft:hidden_convexity_function}).
    More generally, letting $\UNCERTAIN$ be a set and
 \( \correspondence \subset \UNCERTAIN \times \RR^d \) 
be a correspondence between $\UNCERTAIN$ and~$\RR^d$,
  we define the coupling~$\Fenchelcoupling_{\correspondence}: \UNCERTAIN \times \RR^d \to \barRR $
by \( \Fenchelcoupling_{\correspondence}\np{\uncertain, \dual} 
= \sup_{\primal\in\uncertain\correspondence}\FenchelCoupling{\primal}{\dual} \),
for any \( \np{\uncertain, \dual} \in \UNCERTAIN\times\RR^d \).
Then, an easy computation shows that 
the $\Fenchelcoupling_{\correspondence}$-Fenchel-Moreau conjugate
\( \SFM{\fonctionuncertain}{\Fenchelcoupling_{\correspondence}} \) of a
function \( \fonctionuncertain : \UNCERTAIN \to \barRR \)
can be expressed as the Fenchel conjugate of the
conditional infimum~\( \InfimalPostComposition{\correspondence}{\fonctionuncertain} \):
  \begin{equation}
    \SFM{\fonctionuncertain}{\Fenchelcoupling_{\correspondence}}=
    \SFM{ \bp{\ConditionalInfimum{\correspondence}{\fonctionuncertain}} }{\Fenchelcoupling}
    \eqfinp 
  \end{equation}

  We can use the conditional infimum
  in Proposition~\ref{pr:inf_ConditionalInfimum_inf_original}
  to obtain lower bound convex programs for nonconvex problems as follows.
  Consider, on the one hand, a function \( \fonctionuncertain : \UNCERTAIN \to \barRR \)
  and a subset \( \Uncertain \subset \UNCERTAIN \),
  and, on the other hand,
  a correspondence~$\correspondence$ on~$\UNCERTAIN\times\RR^d$, 
  and a convex subset \( \Primal \subset \RR^d \).
  If \( \correspondence\Primal \subset \Uncertain \), we have the inequality
  \begin{equation}
   \inf_{\uncertain \in \Uncertain} \fonctionuncertain\np{\uncertain} 
    \geq 
     \underbrace{ \inf_{\primal \in \Primal}
       \LFMbi{\bp{\nInfCond{\fonctionuncertain}{\correspondence}}}\np{\primal} }%
     _{\textrm{lower bound convex program}}
\eqfinp      
\end{equation}
This may be interesting when \( \LFMbi{\fonctionuncertain} \)
is trivial, but
\( \LFMbi{\bp{\nInfCond{\fonctionuncertain}{\correspondence}}} \)
is not, like when \( \fonctionuncertain \) is the
\lzeropseudonorm\ on~$\RR^d$ and $\correspondence$ is
given by the normalization mapping onto the Euclidean sphere
\cite{Chancelier-DeLara:2021_ECAPRA_JCA}.
\bigskip

\textbf{Acknowledgements.} This note was inspired by the talk
given by Marc Teboulle at the One World Optimization Seminar on 
20 April 2020.
At this occasion, we discovered the paper~\cite{BenTal-Ben-Teboulle:1996} and
the use of the vocable ``hidden convexity''\footnote{%
  We have used the same expression ``hidden convexity'' in the title
  of~\cite{Chancelier-DeLara:2021_ECAPRA_JCA}
  without knowing the paper~\cite{BenTal-Ben-Teboulle:1996} at the time.}.

\newcommand{\noopsort}[1]{} \ifx\undefined\allcaps\def\allcaps#1{#1}\fi

\end{document}